\makeatletter

\documentclass[
    11pt,
    a4paper,
    oneside,
    openright,
    center,
    chapterbib,
    crosshair,
    fleqn,
    headcount,
    headline,
    indent,
    indentfirst=false,
    portrait,
    phonetic,
    oldernstyle,
    onecolumn,
    sfbold,
    upper,
]{amsart}

\let\th@plain\relax

\PassOptionsToPackage{T2A}{fontenc} % T1,OT1,T2A,OT2
\PassOptionsToPackage{utf8}{inputenc} % utf8
\PassOptionsToPackage{cyrpart}{cyrillic}
\PassOptionsToPackage{british,ngerman,russian}{babel}
\PassOptionsToPackage{
    english,
    ngerman,
    russian,
    capitalise,
}{cleveref}
\PassOptionsToPackage{
    margin=10pt,
    position=bottom,
    justification=centering,
    font=small,
    labelformat=simple,
    labelfont={sc},
    labelsep={period},
    textfont={it}
}{caption}
\PassOptionsToPackage{
    margin=10pt,
    position=bottom,
    justification=centering,
    font=small,
    labelformat=parens,
    labelfont={bf},
    labelsep={space},
    textfont={it}
}{subcaption}
\PassOptionsToPackage{framemethod=TikZ}{mdframed}
\PassOptionsToPackage{
    bookmarks=true,
    bookmarksopen=false,
    bookmarksopenlevel=0,
    bookmarkstype=toc,
    colorlinks=false,
    raiselinks=true,
    hyperfigures=true,
}{hyperref}
\PassOptionsToPackage{normalem}{ulem}
\PassOptionsToPackage{
    amsmath,
    thmmarks,
}{ntheorem}
\PassOptionsToPackage{table}{xcolor}
\PassOptionsToPackage{
    all,
    color,
    curve,
    frame,
    import,
    knot,
    line,
    movie,
    rotate,
    textures,
    tile,
    tips,
    web,
    xdvi,
}{xy}
\PassOptionsToPackage{
    reset,
    left=1in,
    right=1in,
    top=20mm,
    bottom=20mm,
    heightrounded,
}{geometry}
\PassOptionsToPackage{
    symbol*, % use symbols to number the footnotes.
    multiple, % comma-separates symbols on multiply footnoted words.
}{footmisc}
\PassOptionsToPackage{overload}{textcase}
\PassOptionsToPackage{indentafter}{titlesec}

\usepackage{amsfonts}
\usepackage{amsmath}
\usepackage{amssymb}
\usepackage{ntheorem} % <— muss nach den ams* Packages vorkommen!!
\usepackage{array}
\usepackage{babel}
\usepackage{bbding}
\usepackage{bbm}
\usepackage{bibentry}
\usepackage{booktabs}
\usepackage{bold-extra} % <- needed for bold+scshape
\usepackage{calc}
\usepackage{cancel}
\usepackage{caption} %% bevor [subcaption]
\usepackage{changepage}
\usepackage{cjhebrew}
\usepackage{cmlgc}
\usepackage{colonequals}
\usepackage{color}
\usepackage{comment}
\usepackage{datetime}
\usepackage{dsfont}
\usepackage{etex}
\usepackage{etoolbox}
\usepackage{eurosym}
\usepackage{fancybox}
\usepackage{fancyhdr}
\usepackage{float}
\usepackage{fontenc}
\usepackage{footmisc}
\usepackage{fp}
\usepackage{geometry}
\usepackage{graphicx}
\usepackage{ifpdf}
\usepackage{ifthen}
\usepackage{ifoddpage}
\usepackage{ifnextok}  % funktioniert unter BEAMERMODUS nicht
\usepackage{index}     % funktioniert unter BEAMERMODUS nicht
\usepackage{inputenc}
\usepackage{latexsym}
\usepackage{lineno}
\usepackage{listings}
\usepackage{longtable}
\usepackage{lscape}
\usepackage{mathrsfs}
\usepackage{multicol}
\usepackage{multirow}
\usepackage{nameref}
\usepackage{nowtoaux}
\usepackage{paralist}
\usepackage{enumerate} % nach [paralist]
\usepackage{pgf}
\usepackage{pgfplots}
\usepackage{phonetic}
\usepackage{proof}
\usepackage{qtree}
\usepackage{refcount}
\usepackage{savesym}
\usepackage{stmaryrd}
\usepackage{synttree}
\usepackage{subcaption}
\usepackage{suffix}
\usepackage{yfonts} % <— Altgotische Fonts
\usepackage{textcase} % <- causes issues with amsart
\usepackage{tikz}
\usepackage{xy}
\usepackage{wrapfig}
\usepackage{xcolor}
\usepackage{xspace}
\usepackage{xstring}
\usepackage{hyperref}
\usepackage{arydshln}
\usepackage{cleveref} % must vor hyperref geladen werden.

\pgfplotsset{compat=newest}
\usetikzlibrary{math}

\usetikzlibrary{
    angles,
    arrows,
    automata,
    calc,
    decorations,
    decorations.pathmorphing,
    decorations.pathreplacing,
    positioning,
    patterns,
    quotes,
}

\savesymbol{Diamond}
\savesymbol{emptyset}
\savesymbol{ggg}
\savesymbol{int}
\savesymbol{lll}
\savesymbol{RectangleBold}
\savesymbol{langle}
\savesymbol{rangle}
\savesymbol{hookrightarrow}
\savesymbol{hookleftarrow}
\savesymbol{Asterisk}
\usepackage{mathabx}
\usepackage{wasysym}

\restoresymbol{x}{Diamond}
\restoresymbol{x}{emptyset}
\restoresymbol{x}{ggg}
\restoresymbol{x}{int}
\restoresymbol{x}{lll}
\restoresymbol{x}{RectangleBold}
\restoresymbol{x}{langle}
\restoresymbol{x}{rangle}
\restoresymbol{x}{hookrightarrow}
\restoresymbol{x}{hookleftarrow}
\restoresymbol{x}{Asterisk}

\ifpdf
    \usepackage{pdfcolmk}
\fi

\usepackage{mdframed}

\DeclareFontFamily{U}{MnSymbolA}{}
\DeclareFontShape{U}{MnSymbolA}{m}{n}{
    <-6> MnSymbolA5
    <6-7> MnSymbolA6
    <7-8> MnSymbolA7
    <8-9> MnSymbolA8
    <9-10> MnSymbolA9
    <10-12> MnSymbolA10
    <12-> MnSymbolA12
}{}
\DeclareFontShape{U}{MnSymbolA}{b}{n}{
    <-6> MnSymbolA-Bold5
    <6-7> MnSymbolA-Bold6
    <7-8> MnSymbolA-Bold7
    <8-9> MnSymbolA-Bold8
    <9-10> MnSymbolA-Bold9
    <10-12> MnSymbolA-Bold10
    <12-> MnSymbolA-Bold12
}{}
\DeclareSymbolFont{MnSyA}{U}{MnSymbolA}{m}{n}
\DeclareMathSymbol{\lcirclearrowright}{\mathrel}{MnSyA}{252}
\DeclareMathSymbol{\lcirclearrowdown}{\mathrel}{MnSyA}{255}
\DeclareMathSymbol{\rcirclearrowleft}{\mathrel}{MnSyA}{250}
\DeclareMathSymbol{\rcirclearrowdown}{\mathrel}{MnSyA}{251}

\DeclareFontFamily{U}{MnSymbolC}{}
\DeclareSymbolFont{MnSyC}{U}{MnSymbolC}{m}{n}
\DeclareFontShape{U}{MnSymbolC}{m}{n}{
    <-6>  MnSymbolC5
    <6-7>  MnSymbolC6
    <7-8>  MnSymbolC7
    <8-9>  MnSymbolC8
    <9-10> MnSymbolC9
    <10-12> MnSymbolC10
    <12->   MnSymbolC12%
}{}
\DeclareMathSymbol{\powerset}{\mathord}{MnSyC}{180}

\DeclareMathAlphabet{\mathpzc}{OT1}{pzc}{m}{it}

\def\boolwahr{true}
\def\boolfalsch{false}
\def\boolleer{}

\let\boolinappendix\boolfalsch
\let\boolinmdframed\boolfalsch
\let\eqtagset\boolfalsch
\let\eqtaglabel\boolleer
\let\eqtagsymb\boolleer

\newcount\bufferctr
\newcount\bufferreplace

\newlength\rtab
\newlength\gesamtlinkerRand
\newlength\gesamtrechterRand
\newlength\ownspaceabovethm
\newlength\ownspacebelowthm
\def\secnumberingpt{.}
\def\secnumberingseppt{.}
\def\subsecnumberingseppt{}
\def\thmnumberingpt{.}
\def\thmnumberingseppt{}
\def\thmForceSepPt{.}

\definecolor{leer}{gray}{1}
\definecolor{boxgrau}{gray}{0.85}
\definecolor{dunkelgrau}{gray}{0.5}
\definecolor{maroon}{rgb}{0.6901961,0.1882353,0.3764706}
\definecolor{dunkelgruen}{rgb}{0.015625,0.363281,0.109375}
\definecolor{dunkelrot}{rgb}{0.5450980392,0,0}
\definecolor{dunkelblau}{rgb}{0,0,0.5450980392}
\definecolor{blau}{rgb}{0,0,1}
\definecolor{newresult}{rgb}{0.6,0.6,0.6}
\definecolor{improvedresult}{rgb}{0.9,0.9,0.9}
\definecolor{hervorheben}{rgb}{0,0.9,0.7}
\definecolor{starkesblau}{rgb}{0.1019607843,0.3176470588,0.8156862745}
\definecolor{achtung}{rgb}{1,0.5,0.5}
\definecolor{frage}{rgb}{0.5,1,0.5}
\definecolor{schreibweise}{rgb}{0,0.7,0.9}
\definecolor{axiom}{rgb}{0,0.3,0.3}

\def\let@name#1#2{
    \expandafter\let\csname #1\expandafter\endcsname\csname #2\endcsname\relax
}
\DeclareRobustCommand\crfamily{\fontfamily{ccr}\selectfont}
\DeclareTextFontCommand{\textcr}{\crfamily}

\def\ifthenelseleer#1#2#3{\ifthenelse{\equal{#1}{}}{#2}{#1#3}}
\def\bedingtesspaceexpand#1#2#3{\ifthenelseleer{\csname #1\endcsname}{#3}{#2#3}}

\def\nvraum{\@ifnextchar\bgroup{\nvraum@c}{\nvraum@bes}}
    \def\nvraum@c#1{\vspace*{-#1\baselineskip}}
    \def\nvraum@bes{\vspace*{-\baselineskip}}
\def\erlaubeplatz{\relax\ifmmode\else\@\xspace\fi}
\def\entferneplatz{\relax\ifmmode\else\expandafter\@gobble\fi}

\def\send@toaux#1{\@bsphack\protected@write\@auxout{}{\string#1}\@esphack}

\def\rlabel#1[#2]#3#4#5{#5\rlabel@aux{#1}[#2]{#3}{#4}{#5}}
    \def\rlabel@aux#1[#2]#3#4#5{%
        \send@toaux{\newlabel{#1}{{\@currentlabel}{\thepage}{{\unexpanded{#5}}}{#2.\csname the#2\endcsname}{}}}\relax%
    }

\def\tag@rawscheme#1#2[#3]#4#5{\@ifnextchar[{\tag@rawscheme@{#1}{#2}[#3]{#4}{#5}}{\tag@rawscheme@{#1}{#2}[#3]{#4}{#5}[*]}}
    \def\tag@rawscheme@#1#2[#3]#4#5[#6]{\@ifnextchar\bgroup{\tag@rawscheme@@{#1}{#2}[#3]{#4}{#5}[#6]}{\tag@rawscheme@@{#1}{#2}[#3]{#4}{#5}[#6]{}}}
    \def\tag@rawscheme@@#1#2[#3]#4#5[#6]#7{%
        \ifthenelse{\equal{#6}{*}}{%
            \ifthenelse{\equal{#7}{\boolleer}}{\refstepcounter{#3}#4\csname the#3\endcsname#5}{#4#7#5}%
        }{%
            \refstepcounter{#3}#4%
            \ifthenelse{\equal{#7}{\boolleer}}{\rlabel{#6}[#3]{#1}{#2}{\csname the#3\endcsname}}{\rlabel{#6}[#3]{#1}{#2}{#7}}%
            #5%
        }%
    }
\def\tag@scheme#1#2[#3]{\tag@rawscheme{#1}{#2}[#3]{\upshape(}{\upshape)}}

\def\eqtag@post#1{\makebox[0pt][r]{#1}}
\def\eqtag@pre{\tag@scheme{Eq}{Equation}[Xe]}
\def\eqtag{\@ifnextchar[{\eqtag@}{\eqtag@[*]}}
    \def\eqtag@[#1]{\@ifnextchar\bgroup{\eqtag@@[#1]}{\eqtag@@[#1]{}}}
    \def\eqtag@@[#1]#2{\eqtag@post{\eqtag@pre[#1]{#2}}}

\def\eqcref#1{\text{(\ref{#1})}}

\def\punktlabel#1{\label{it:#1:\beweislabel}}
\def\punktcref#1{\eqcref{it:#1:\beweislabel}}

\def\opfromto[#1]_#2^#3{\underset{#2}{\overset{#3}{#1}}}
\def\textoverset#1#2{\overset{\text{#1}}{#2}}

\def\eqcrefoverset#1#2{\textoverset{\eqcref{#1}}{#2}}

\def\mathclap#1{#1}
\def\oberunterset#1{\@ifnextchar^{\oberunterset@oben{#1}}{\oberunterset@unten{#1}}}
    \def\oberunterset@oben#1^#2_#3{\underset{\mathclap{#3}}{\overset{\mathclap{#2}}{#1}}}
    \def\oberunterset@unten#1_#2^#3{\underset{\mathclap{#2}}{\overset{\mathclap{#3}}{#1}}}
    \def\breitunderbrace#1_#2{\underbrace{#1}_{\mathclap{#2}}}
    \def\breitoverbrace#1^#2{\overbrace{#1}^{\mathclap{#2}}}
    \def\breitunderbracket#1_#2{\underbracket{#1}_{\mathclap{#2}}}
    \def\breitoverbracket#1^#2{\overbracket{#1}^{\mathclap{#2}}}

\def\generatenestedsecnumbering#1#2#3{%
    \expandafter\gdef\csname thelong#3\endcsname{%
        \expandafter\csname the#2\endcsname%
        \secnumberingpt%
        \expandafter\csname #1\endcsname{#3}%
    }%
    \expandafter\gdef\csname theshort#3\endcsname{%
        \expandafter\csname #1\endcsname{#3}%
    }%
}
\def\generatenestedthmnumbering#1#2#3{%
    \expandafter\gdef\csname the#3\endcsname{%
        \expandafter\csname the#2\endcsname%
        \thmnumberingpt%
        \expandafter\csname #1\endcsname{#3}%
    }%
    \expandafter\gdef\csname theshort#3\endcsname{%
        \expandafter\csname #1\endcsname{#3}%
    }%
}

\def\+#1{\addtocounter{#1}{1}}
\def\setcounternach#1#2{\setcounter{#1}{#2}\addtocounter{#1}{-1}}

\def\forcepunkt#1{#1\IfEndWith{#1}{.}{}{.}}
\def\lateinabkuerzung#1#2{%
    \expandafter\gdef\csname #1\endcsname{\emph{#2}\@ifnextchar.{\entferneplatz}{\erlaubeplatz}}
}
\def\deutscheabkuerzung#1#2{%
    \expandafter\gdef\csname #1\endcsname{{#2}\@ifnextchar.{\entferneplatz}{\erlaubeplatz}}
}

\def\matrix#1{\left(\begin{array}{#1}}
    \def\endmatrix{\end{array}\right)}
\def\smatrix{\left(\begin{smallmatrix}}
    \def\endsmatrix{\end{smallmatrix}\right)}

\def\multiargrekursiverbefehl#1#2#3#4#5#6#7#8{%
    \expandafter\gdef\csname#1\endcsname #2##1#4{\csname #1@anfang\endcsname##1#3\egroup}
    \expandafter\def\csname #1@anfang\endcsname##1#3{#5##1\@ifnextchar\egroup{\csname #1@ende\endcsname}{#7\csname #1@mitte\endcsname}}
    \expandafter\def\csname #1@mitte\endcsname##1#3{#6##1\@ifnextchar\egroup{\csname #1@ende\endcsname}{#7\csname #1@mitte\endcsname}}
    \expandafter\def\csname #1@ende\endcsname##1{#8}
}
\multiargrekursiverbefehl{svektor}{[}{;}{]}{\begin{smatrix}}{}{\\}{\\\end{smatrix}}
\multiargrekursiverbefehl{vektor}{[}{;}{]}{\begin{matrix}{c}}{}{\\}{\\\end{matrix}}
\multiargrekursiverbefehl{vektorzeile}{}{,}{;}{}{&}{}{}
\multiargrekursiverbefehl{matlabmatrix}{[}{;}{]}{\begin{smatrix}\vektorzeile}{\vektorzeile}{;\\}{;\end{smatrix}}

\def\faelle[#1]#2{\left\{\begin{array}[#1]{#2}}
    \def\endfaelle{\end{array}\right.}

\def\BeweisRichtung[#1]{\@ifnextchar\bgroup{\@BeweisRichtung@c[#1]}{\@BeweisRichtung@bes[#1]}}
    \def\@BeweisRichtung@bes[#1]{{\bfseries(#1).~}}
    \def\@BeweisRichtung@c[#1]#2#3{{\bfseries(#2 #1 #3).~}}
\def\erzeugeBeweisRichtungBefehle#1#2{
    \expandafter\gdef\csname #1text\endcsname##1##2{\BeweisRichtung[#2]{##1}{##2}}
    \expandafter\gdef\csname #1\endcsname{%
        \@ifnextchar\bgroup{\csname #1@\endcsname}{\csname #1text\endcsname{}{}}%
    }
    \expandafter\gdef\csname #1@\endcsname##1##2{%
        \csname #1text\endcsname{\punktcref{##1}}{\punktcref{##2}}%
    }
}
\erzeugeBeweisRichtungBefehle{hinRichtung}{$\Rightarrow$}
\erzeugeBeweisRichtungBefehle{herRichtung}{$\Leftarrow$}
\erzeugeBeweisRichtungBefehle{hinherRichtung}{$\Leftrightarrow$}

\def\cal#1{\mathcal{#1}}
\def\brkt#1{\langle{}#1{}\rangle}
\def\mathfrak#1{\mbox{\usefont{U}{euf}{m}{n}#1}}
\def\kurs#1{\textit{#1}}
\def\rectangleblack{\text{\RectangleBold}}

\def\squareblack{\blacksquare}

\def\crefname@full#1#2#3#4#5{%
    \crefname{#1}{#2}{#3}
    \Crefname{#1}{#4}{#5}
}
\def\crefname@fullmod#1#2#3#4#5{%
    \crefname@full{#1}{#2}{#3}{#4}{#5}
    \crefname@full{#1@basic}{#2}{#3}{#4}{#5}
    \crefname@full{#1@withName}{#2}{#3}{#4}{#5}
}
\crefname@full{chapter}{chapter}{chapters}{Chapter}{Chapters}
\crefname@full{appendix}{appendix}{appendices}{Appendix}{Appendices}
\crefname@full{section}{section}{sections}{Section}{Sections}
\crefname@full{subsection}{section}{sections}{Section}{Sections}
\crefname@full{subsubsection}{section}{sections}{Section}{Sections}
\crefname@full{subsubsubsection}{section}{sections}{Section}{Sections}
\crefname@full{table}{table}{tables}{Table}{Tables}
\crefname@full{figure}{figure}{figures}{Figure}{Figures}
\crefname@full{subfigure}{figure}{figures}{Figure}{Figures}

\crefname@fullmod{thm}{theorem}{theorems}{Theorem}{Theorems}
\crefname@fullmod{thmStar}{theorem}{theorems}{Theorem}{Theorems}
\crefname@fullmod{conj}{conjecture}{conjectures}{Conjecture}{Conjectures}
\crefname@fullmod{cor}{corollary}{corollaries}{Corollary}{Corollaries}
\crefname@fullmod{defn}{definition}{definitions}{Definition}{Definitions}
\crefname@fullmod{conv}{convention}{conventions}{Convention}{Conventions}
\crefname@fullmod{e.g.}{example}{examples}{Example}{Examples}
\crefname@fullmod{prop}{proposition}{propositions}{Proposition}{Propositions}
\crefname@fullmod{proof}{proof}{proofs}{Proof}{Proofs}
\crefname@fullmod{lemm}{lemma}{lemmata}{Lemma}{Lemmata}
\crefname@fullmod{qstn}{question}{questions}{Question}{Questions}
\crefname@fullmod{rem}{remark}{remarks}{Remark}{Remarks}

    \def\qedEIGEN#1{\@ifnextchar[{\qedEIGEN@c{#1}}{\qedEIGEN@bes{#1}}}%]
    \def\qedEIGEN@bes#1{%
        \parfillskip=0pt%            % so \par doesnt push \square to left
        \widowpenalty=10000%         % so we dont break the page before \square
        \displaywidowpenalty=10000%  % ditto
        \finalhyphendemerits=0%      % TeXbook exercise 14.32
        \leavevmode%                 % \nobreak means lines not pages
        \unskip%                     % remove previous space or glue
        \nobreak%                    % don’t break lines
        \hfil%                       % ragged right if we spill over
        \penalty50%                  % discouragement to do so
        \hskip.2em%                  % ensure some space
        \null%                       % anchor following \hfill
        \hfill%                      % push \square to right
        #1%                          % the end-of-proof mark
        \par%
    }
    \def\qedEIGEN@c#1[#2]{%
        \parfillskip=0pt%            % so \par doesnt push \square to left
        \widowpenalty=10000%         % so we dont break the page before \square
        \displaywidowpenalty=10000%  % ditto
        \finalhyphendemerits=0%      % TeXbook exercise 14.32
        \leavevmode%                 % \nobreak means lines not pages
        \unskip%                     % remove previous space or glue
        \nobreak%                    % don’t break lines
        \hfil%                       % ragged right if we spill over
        \penalty50%                  % discouragement to do so
        \hskip.2em%                  % ensure some space
        \null%                       % anchor following \hfill
        \hfill%                      % push \square to right
        {#1~{\small\bfseries\upshape (#2)}}%
        \par%
    }
    \def\qedVARIANT#1#2{
        \expandafter\def\csname ennde#1Sign\endcsname{#2}
        \expandafter\def\csname ennde#1\endcsname{\@ifnextchar[{\qedEIGEN@c{#2}}{\qedEIGEN@bes{#2}}} %]
    }
    \qedVARIANT{OfProof}{$\squareblack$}
    \qedVARIANT{OfWork}{\rectangleblack}
    \qedVARIANT{OfSomething}{$\dashv$}
    \qedVARIANT{OnNeutral}{} % \lozenge \bigcirc \blacklozenge

    \def\ra@pretheoremwork{
        \setlength{\theorempreskipamount}{\ownspaceabovethm}
    }
    \def\rathmtransfer#1#2{
        \expandafter\def\csname #2\endcsname{\csname #1\endcsname}
        \expandafter\def\csname end#2\endcsname{\csname end#1\endcsname}
    }

    \def\ranewthm#1#2#3[#4]{
        \theoremstyle{\current@theoremstyle}
        \theoremseparator{\current@theoremseparator}
        \theoremprework{\ra@pretheoremwork}
        \@ifundefined{#1@basic}{\newtheorem{#1@basic}[#4]{#2}}{\renewtheorem{#1@basic}[#4]{#2}}
        \theoremstyle{\current@theoremstyle}
        \theoremseparator{\thmForceSepPt}
        \theoremprework{\ra@pretheoremwork}
        \@ifundefined{#1@withName}{\newtheorem{#1@withName}[#4]{#2}}{\renewtheorem{#1@withName}[#4]{#2}}
        \theoremstyle{nonumberplain}
        \theoremseparator{\thmForceSepPt}
        \theoremprework{\ra@pretheoremwork}
        \@ifundefined{#1@star@basic}{\newtheorem{#1@star@basic}[#4]{#2}}{\renewtheorem{#1@star@basic}[#4]{#2}}
        \theoremstyle{nonumberplain}
        \theoremseparator{\thmForceSepPt}
        \theoremprework{\ra@pretheoremwork}
        \@ifundefined{#1@star@withName}{\newtheorem{#1@star@withName}[#4]{#2}}{\renewtheorem{#1@star@withName}[#4]{#2}}
        \umbauenenv{#1}{#3}[#4]
        \umbauenenv{#1@star}{#3}[#4]
        \rathmtransfer{#1@star}{#1*}
    }

    \def\umbauenenv#1#2[#3]{%
        \expandafter\def\csname #1\endcsname{\relax%
            \@ifnextchar[{\csname #1@\endcsname}{\csname #1@\endcsname[*]}%
        }
        \expandafter\def\csname #1@\endcsname[##1]{\relax%
            \@ifnextchar[{\csname #1@@\endcsname[##1]}{\csname #1@@\endcsname[##1][*]}%
        }
        \expandafter\def\csname #1@@\endcsname[##1][##2]{%
            \ifx*##1%
                \def\enndeOfBlock{\csname end#1@basic\endcsname}
                \csname #1@basic\endcsname%
            \else%
                \def\enndeOfBlock{\csname end#1@withName\endcsname}
                \csname #1@withName\endcsname[##1]%
            \fi%
            \def\makelabel####1{%
                \gdef\beweislabel{####1}%
                \label{\beweislabel}%
            }%
            \ifx*##2%
                \def\enndeSymbol{\qedEIGEN{#2}}
            \else%
                \def\enndeSymbol{\qedEIGEN{#2}[##2]}
            \fi
        }
        \expandafter\gdef\csname end#1\endcsname{\enndeSymbol\enndeOfBlock}
    }

        \def\current@theoremstyle{plain}
        \def\current@theoremseparator{\thmnumberingseppt}
        \theoremstyle{\current@theoremstyle}
        \theoremseparator{\current@theoremseparator}
        \theoremsymbol{}

    \generatenestedthmnumbering{arabic}{section}{X}
    \generatenestedthmnumbering{arabic}{section}{Xe}
    \generatenestedthmnumbering{Roman}{section}{Xsp}

        \theoremheaderfont{\upshape\bfseries}
        \theorembodyfont{\slshape}

    \ranewthm{thm}{Theorem}{\enndeOnNeutralSign}[X]
    \ranewthm{lemm}{Lemma}{\enndeOnNeutralSign}[X]
    \ranewthm{cor}{Corollary}{\enndeOnNeutralSign}[X]
    \ranewthm{prop}{Proposition}{\enndeOnNeutralSign}[X]

        \theorembodyfont{\upshape}

    \ranewthm{defn}{Definition}{\enndeOnNeutralSign}[X]
    \ranewthm{conv}{Convention}{\enndeOnNeutralSign}[X]
    \ranewthm{e.g.}{Example}{\enndeOnNeutralSign}[X]
    \ranewthm{fact}{Fact}{\enndeOnNeutralSign}[X]
    \ranewthm{rem}{Remark}{\enndeOnNeutralSign}[X]
    \ranewthm{qstn}{Question}{\enndeOnNeutralSign}[X]

        \theoremheaderfont{\itshape\bfseries}
        \theorembodyfont{\upshape}

    \ranewthm{proof@tmp}{Proof}{\enndeOfProofSign}[Xdisplaynone]
    \rathmtransfer{proof@tmp*}{proof}

    \def\behauptungbeleg@claim{%
        \iflanguage{british}{Claim}{%
        \iflanguage{english}{Claim}{%
        \iflanguage{ngerman}{Behauptung}{%
        \iflanguage{russian}{Утверждение}{%
        Claim%
        }}}}%
    }
    \def\behauptungbeleg@pf@kurz{%
        \iflanguage{british}{Pf}{%
        \iflanguage{english}{Pf}{%
        \iflanguage{ngerman}{Bew}{%
        \iflanguage{russian}{Доказательство}{%
        Pf%
        }}}}%
    }
    \def\behauptungbeleg{\@ifnextchar\bgroup{\behauptungbeleg@c}{\behauptungbeleg@bes}}
            \def\behauptungbeleg@c#1{\item[{\bfseries \behauptungbeleg@claim\erlaubeplatz #1.}]}
            \def\behauptungbeleg@bes{\item[{\bfseries \behauptungbeleg@claim.}]}
        \def\belegbehauptung{\item[{\bfseries\itshape\behauptungbeleg@pf@kurz.}]}

    \newdateformat{standardshort}{\oldstylenums{\THEYEAR}.\oldstylenums{\THEMONTH}.\oldstylenums{\THEDAY}}
    \newdateformat{standardcompact}{\THEYEAR\twodigit{\THEMONTH}\twodigit{\THEDAY}}
    \newdateformat{standardlong}{\THEYEAR\ \monthname\ \THEDAY}
    \newcolumntype{\RECHTS}[1]{>{\raggedleft}p{#1}}
    \newcolumntype{\LINKS}[1]{>{\raggedright}p{#1}}
    \newcolumntype{m}{>{$}l<{$}}
    \newcolumntype{C}{>{$}c<{$}}
    \newcolumntype{L}{>{$}l<{$}}
    \newcolumntype{R}{>{$}r<{$}}
    \newcolumntype{0}{@{\hspace{0pt}}}
    \newcolumntype{\LINKSRAND}{@{\hspace{\@totalleftmargin}}}
    \newcolumntype{h}{@{\extracolsep{\fill}}}
    \newcolumntype{i}{>{\itshape}}
    \newcolumntype{t}{@{\hspace{\tabcolsep}}}
    \newcolumntype{q}{@{\hspace{1em}}}
    \newcolumntype{n}{@{\hspace{-\tabcolsep}}}
    \newcolumntype{M}[2]{%
        >{\begin{minipage}{#2}\begin{math}}%
        {#1}%
        <{\end{math}\end{minipage}}%
    }
    \newcolumntype{T}[2]{%
        >{\begin{minipage}{#2}}%
        {#1}%
        <{\end{minipage}}%
    }
    \def\punkteumgebung@genbefehl#1#2#3{
        \punkteumgebung@genbefehl@{#1}{#2}{#3}{}{}
        \punkteumgebung@genbefehl@{multi#1}{#2}{#3}{
            \setlength{\columnsep}{10pt}%
            \setlength{\columnseprule}{0pt}%
            \begin{multicols}{\thecolumnanzahl}%
        }{\end{multicols}\nvraum{1}}
    }
    \def\punkteumgebung@genbefehl@#1#2#3#4#5{
        \expandafter\gdef\csname #1\endcsname{
            \@ifnextchar\bgroup{\csname #1@c\endcsname}{\csname #1@bes\endcsname}
        }%]
            \expandafter\def\csname #1@c\endcsname##1{
                \@ifnextchar[{\csname #1@c@\endcsname{##1}}{\csname #1@c@\endcsname{##1}[\z@]}
            }%]
            \expandafter\def\csname #1@c@\endcsname##1[##2]{
                \@ifnextchar[{\csname #1@c@@\endcsname{##1}[##2]}{\csname #1@c@@\endcsname{##1}[##2][\z@]}
            }%]
            \expandafter\def\csname #1@c@@\endcsname##1[##2][##3]{
                \let\alterlinkerRand\gesamtlinkerRand
                \let\alterrechterRand\gesamtrechterRand
                \addtolength{\gesamtlinkerRand}{##2}
                \addtolength{\gesamtrechterRand}{##3}
                \advance\linewidth -##2%
                \advance\linewidth -##3%
                \advance\@totalleftmargin ##2%
                \parshape\@ne \@totalleftmargin\linewidth%
                #4
                \begin{#2}[\upshape ##1]%
                    \setlength{\parskip}{0.5\baselineskip}\relax%
                    \setlength{\topsep}{\z@}\relax%
                    \setlength{\partopsep}{\z@}\relax%
                    \setlength{\parsep}{\parskip}\relax%
                    \setlength{\itemsep}{#3}\relax%
                    \setlength{\listparindent}{\z@}\relax%
                    \setlength{\itemindent}{\z@}\relax%
            }
            \expandafter\def\csname #1@bes\endcsname{
                \@ifnextchar[{\csname #1@bes@\endcsname}{\csname #1@bes@\endcsname[\z@]}
            }%]
            \expandafter\def\csname #1@bes@\endcsname[##1]{
                \@ifnextchar[{\csname #1@bes@@\endcsname[##1]}{\csname #1@bes@@\endcsname[##1][\z@]}
            }%]
            \expandafter\def\csname #1@bes@@\endcsname[##1][##2]{
                \let\alterlinkerRand\gesamtlinkerRand
                \let\alterrechterRand\gesamtrechterRand
                \addtolength{\gesamtlinkerRand}{##1}
                \addtolength{\gesamtrechterRand}{##2}
                \advance\linewidth -##1%
                \advance\linewidth -##2%
                \advance\@totalleftmargin ##1%
                \parshape\@ne \@totalleftmargin\linewidth%
                #4
                \begin{#2}%
                    \setlength{\parskip}{0.5\baselineskip}\relax%
                    \setlength{\topsep}{\z@}\relax%
                    \setlength{\partopsep}{\z@}\relax%
                    \setlength{\parsep}{\parskip}\relax%
                    \setlength{\itemsep}{#3}\relax%
                    \setlength{\listparindent}{\z@}\relax%
                    \setlength{\itemindent}{\z@}\relax%
            }
        \expandafter\gdef\csname end#1\endcsname{%
            \end{#2}#5
            \setlength{\gesamtlinkerRand}{\alterlinkerRand}
            \setlength{\gesamtlinkerRand}{\alterrechterRand}
        }
    }

    \def\ritempunkt{{\Large \textbullet}} % \textbullet, $\sqbullet$, $\blacktriangleright$
    \setdefaultitem{\ritempunkt}{\ritempunkt}{\ritempunkt}{\ritempunkt}
    \punkteumgebung@genbefehl{itemise}{compactitem}{\parskip}{}{}
    \punkteumgebung@genbefehl{kompaktitem}{compactitem}{\z@}{}{}
    \punkteumgebung@genbefehl{enumerate}{compactenum}{\parskip}{}{}
    \punkteumgebung@genbefehl{kompaktenum}{compactenum}{\z@}{}{}

    \renewenvironment{thebibliography}[1]{%
        \begin{ALTthebibliography}{#1}
        \addcontentsline{toc}{part}{\bibname}
    }{%
        \end{ALTthebibliography}
    }

\def\displaysum_#1{\@ifnextchar^{\displaysum@both_{#1}}{\displaysum@@sub{#1}}}
    \def\displaysum@both_#1^#2{\displaysum@@subsup{#1}{#2}}
    \def\displaysum@@sub#1{\mathop{\displaystyle\csname sum\endcsname_{#1}}}
    \def\displaysum@@subsup#1#2{\mathop{\displaystyle\csname sum\endcsname_{#1}^{#2}}}
\def\displaysup_#1{\@ifnextchar^{\displaysup@both_{#1}}{\displaysup@@sub{#1}}}
    \def\displaysup@both_#1^#2{\displaysup@@subsup{#1}{#2}}
    \def\displaysup@@sub#1{\mathop{\displaystyle\csname sup\endcsname_{#1}}}
    \def\displaysup@@subsup#1#2{\mathop{\displaystyle\csname sup\endcsname_{#1}^{#2}}}
\def\displaymin_#1{\@ifnextchar^{\displaymin@both_{#1}}{\displaymin@@sub{#1}}}
    \def\displaymin@both_#1^#2{\displaymin@@subsup{#1}{#2}}
    \def\displaymin@@sub#1{\mathop{\displaystyle\csname min\endcsname_{#1}}}
    \def\displaymin@@subsup#1#2{\mathop{\displaystyle\csname min\endcsname_{#1}^{#2}}}
\def\displaymax_#1{\@ifnextchar^{\displaymax@both_{#1}}{\displaymax@@sub{#1}}}
    \def\displaymax@both_#1^#2{\displaymax@@subsup{#1}{#2}}
    \def\displaymax@@sub#1{\mathop{\displaystyle\csname max\endcsname_{#1}}}
    \def\displaymax@@subsup#1#2{\mathop{\displaystyle\csname max\endcsname_{#1}^{#2}}}
\def\displaylim_#1{\@ifnextchar^{\displaylim@both_{#1}}{\displaylim@@sub{#1}}}
    \def\displaylim@both_#1^#2{\displaylim@@subsup{#1}{#2}}
    \def\displaylim@@sub#1{\mathop{\displaystyle\csname lim\endcsname_{#1}}}
    \def\displaylim@@subsup#1#2{\mathop{\displaystyle\csname lim\endcsname_{#1}^{#2}}}
\def\displayliminf_#1{\@ifnextchar^{\displayliminf@both_{#1}}{\displayliminf@@sub{#1}}}
    \def\displayliminf@both_#1^#2{\displayliminf@@subsup{#1}{#2}}
    \def\displayliminf@@sub#1{\mathop{\displaystyle\csname liminf\endcsname_{#1}}}
    \def\displayliminf@@subsup#1#2{\mathop{\displaystyle\csname liminf\endcsname_{#1}^{#2}}}
\def\displaylimsup_#1{\@ifnextchar^{\displaylimsup@both_{#1}}{\displaylimsup@@sub{#1}}}
    \def\displaylimsup@both_#1^#2{\displaylimsup@@subsup{#1}{#2}}
    \def\displaylimsup@@sub#1{\mathop{\displaystyle\csname limsup\endcsname_{#1}}}
    \def\displaylimsup@@subsup#1#2{\mathop{\displaystyle\csname limsup\endcsname_{#1}^{#2}}}

    \def\matrix#1{\left(\begin{array}[mc]{#1}}
        \def\endmatrix{\end{array}\right)}
    \def\smatrix{\left(\begin{smallmatrix}}
        \def\endsmatrix{\end{smallmatrix}\right)}

    \def\multiargrekursiverbefehl#1#2#3#4#5#6#7#8{%
        \expandafter\gdef\csname#1\endcsname #2##1#4{\csname #1@anfang\endcsname##1#3\egroup}
        \expandafter\def\csname #1@anfang\endcsname##1#3{#5##1\@ifnextchar\egroup{\csname #1@ende\endcsname}{#7\csname #1@mitte\endcsname}}
        \expandafter\def\csname #1@mitte\endcsname##1#3{#6##1\@ifnextchar\egroup{\csname #1@ende\endcsname}{#7\csname #1@mitte\endcsname}}
        \expandafter\def\csname #1@ende\endcsname##1{#8}
    }
    \multiargrekursiverbefehl{svektor}{[}{;}{]}{\begin{smatrix}}{}{\\}{\\\end{smatrix}}
    \multiargrekursiverbefehl{vektor}{[}{;}{]}{\begin{matrix}{c}}{}{\\}{\\\end{matrix}}
    \multiargrekursiverbefehl{vektorzeile}{}{,}{;}{}{&}{}{}
    \multiargrekursiverbefehl{matlabmatrix}{[}{;}{]}{\begin{smatrix}\vektorzeile}{\vektorzeile}{;\\}{;\end{smatrix}}

    \def\underbracenodisplay#1{%
        \mathop{\vtop{\m@th\ialign{##\crcr
        $\hfil\displaystyle{#1}\hfil$\crcr
        \noalign{\kern3\p@\nointerlineskip}%
        \upbracefill\crcr\noalign{\kern3\p@}}}}\limits%
    }

    \def\mathe[#1]#2{%
        \ifthenelse{\equal{\boolinmdframed}{\boolwahr}}{}{\begin{escapeeinzug}}
        \noindent%
        \let\eqtagset\boolfalsch
        \let\eqtaglabel\boolleer
        \let\eqtagsymb\boolleer
        \let\alteqtag\eqtag
        \def\eqtag{\@ifnextchar[{\eqtag@loc@}{\eqtag@loc@[*]}}%
        \def\eqtag@loc@[##1]{\@ifnextchar\bgroup{\eqtag@loc@@[##1]}{\eqtag@loc@@[##1]{}}}%
        \def\eqtag@loc@@[##1]##2{%
            \gdef\eqtagset{\boolwahr}
            \gdef\eqtaglabel{##1}
            \gdef\eqtagsymb{##2}
        }%
        \def\verticalalign{}%
            \IfBeginWith{#1}{t}{\def\verticalalign{t}}{}%
            \IfBeginWith{#1}{m}{\def\verticalalign{c}}{}%
            \IfBeginWith{#1}{b}{\def\verticalalign{b}}{}%
        \def\horizontalalign{\null\hfill\null}%
            \IfEndWith{#1}{l}{}{\null\hfill\null}%
            \IfEndWith{#1}{r}{\def\horizontalalign{}}{}%
        \begin{math}
        \begin{array}[\verticalalign]{0#2}%
    }
        \def\endmathe{%
            \end{array}
            \end{math}\horizontalalign%
            \let\eqtag\alteqtag
            \ifthenelse{\equal{\eqtagset}{\boolwahr}}{\eqtag[\eqtaglabel]{\eqtagsymb}}{}
            \ifthenelse{\equal{\boolinmdframed}{\boolwahr}}{}{\end{escapeeinzug}}%
        }

    \def\longmathe[#1]#2{\relax
        \let\altarraystretch\arraystretch
        \renewcommand\arraystretch{1.2}\relax
        \begin{longtable}[#1]{\LINKSRAND #2}
    }
        \def\endlongmathe{
            \end{longtable}
            \renewcommand\arraystretch{\altarraystretch}
        }

    \def\einzug{\@ifnextchar[{\indents@}{\indents@[\z@]}}%]
        \def\indents@[#1]{\@ifnextchar[{\indents@@[#1]}{\indents@@[#1][\z@]}}%]
        \def\indents@@[#1][#2]{%
            \begin{list}{}{\relax
                \setlength{\topsep}{\z@}\relax
                \setlength{\partopsep}{\z@}\relax
                \setlength{\parsep}{\parskip}\relax
                \setlength{\listparindent}{\z@}\relax
                \setlength{\itemindent}{\z@}\relax
                \setlength{\leftmargin}{#1}\relax
                \setlength{\rightmargin}{#2}\relax
                \let\alterlinkerRand\gesamtlinkerRand
                \let\alterrechterRand\gesamtrechterRand
                \addtolength{\gesamtlinkerRand}{#1}
                \addtolength{\gesamtrechterRand}{#2}
            }\relax
                \item[]\relax
        }
            \def\endeinzug{%
                \setlength{\gesamtlinkerRand}{\alterlinkerRand}
                \setlength{\gesamtlinkerRand}{\alterrechterRand}
                \end{list}%
            }

    \def\escapeeinzug{\begin{einzug}[-\gesamtlinkerRand][-\gesamtrechterRand]}
        \def\endescapeeinzug{\end{einzug}}

    \def\programmiercode{
        \modulolinenumbers[1]
        \begin{einzug}[\rtab][\rtab]%
        \begin{linenumbers}%
            \fontfamily{cmtt}\fontseries{m}\fontshape{u}\selectfont%
            \setlength{\parskip}{1\baselineskip}%
            \setlength{\parindent}{0pt}%
    }
        \def\endprogrammiercode{
            \end{linenumbers}
            \end{einzug}
        }

    \def\schattiertebox@genbefehl#1#2#3{
        \expandafter\gdef\csname #1\endcsname{%
            \@ifnextchar[{\csname #1@args\endcsname}{\csname #1@args\endcsname[#3]}%]%
        }
            \expandafter\def\csname #1@args\endcsname[##1]{%
                \@ifnextchar[{\csname #1@args@l\endcsname[##1]}{\csname #1@args@n\endcsname[##1]}%]%
            }
            \expandafter\def\csname #1@args@l\endcsname[##1][##2]{%
                \@ifnextchar[{\csname #1@args@l@r\endcsname[##1][##2]}{\csname #1@args@l@n\endcsname[##1][##2]}%]%
            }
            \expandafter\def\csname #1@args@n\endcsname[##1]{%
                \let\boolinmdframed\boolwahr
                \begin{mdframed}[#2leftmargin=0,rightmargin=0,outermargin=0,innermargin=0,##1]
            }
            \expandafter\def\csname #1@args@l@n\endcsname[##1][##2]{%
                \let\boolinmdframed\boolwahr
                \begin{mdframed}[#2leftmargin=##2/2,rightmargin=##2/2,outermargin=##2/2,innermargin=##2/2,##1]
            }
            \expandafter\def\csname #1@args@l@r\endcsname[##1][##2][##3]{%
                \let\boolinmdframed\boolwahr
                \begin{mdframed}[#2leftmargin=##2,rightmargin=##3,outermargin=##2,innermargin=##3,##1]
            }
        \expandafter\gdef\csname end#1\endcsname{%
            \end{mdframed}
            \let\boolinmdframed\boolfalsch
        }
    }
        \schattiertebox@genbefehl{schattiertebox}{
            splittopskip=0,%
            splitbottomskip=0,%
            frametitleaboveskip=0,%
            frametitlebelowskip=0,%
            skipabove=1\baselineskip,%
            skipbelow=1\baselineskip,%
            linewidth=2pt,%
            linecolor=black,%
            roundcorner=4pt,%
        }{
            backgroundcolor=leer,%
            nobreak=true,%
        }

        \schattiertebox@genbefehl{schattierteboxdunn}{
            splittopskip=0,%
            splitbottomskip=0,%
            frametitleaboveskip=0,%
            frametitlebelowskip=0,%
            skipabove=1\baselineskip,%
            skipbelow=1\baselineskip,%
            linewidth=1pt,%
            linecolor=black,%
            roundcorner=2pt,%
        }{
            backgroundcolor=leer,%
            nobreak=true,%
        }

    \def\tikzsetzepfeil#1{%
        \begin{tikzpicture}[remember picture,overlay,>=latex]%
            \draw #1;%
        \end{tikzpicture}%
    }
    
    \def\tikzsetzekreise[#1]#2#3{%
        \tikzsetzepfeil{%
        [rounded corners,#1]%
            ([shift={(-\tabcolsep,0.75\baselineskip)}]#2)%
            rectangle%
            ([shift={(\tabcolsep,-0.5\baselineskip)}]#3)
        }%
    }

    \tikzset{
        >=stealth,
        auto,
        node distance=1cm,
        thick,
        main node/.style={
            circle,draw,font=\sffamily\Large\bfseries,minimum size=0pt
        },
        state/.style={minimum size=0pt}
        loop above right/.style={loop,out=30,in=60,distance=0.5cm},
        loop above left/.style={above left,out=150,in=120,loop},
        loop below right/.style={below right,out=330,in=300,loop},
        loop below left/.style={below left,out=240,in=210,loop},
        itria/.style={
            draw,dashed,shape border uses incircle,
            isosceles triangle,shape border rotate=90,yshift=-1.45cm
        },
        rtria/.style={
            draw,dashed,shape border uses incircle,
            isosceles triangle,isosceles triangle apex angle=90,
            shape border rotate=-45,yshift=0.2cm,xshift=0.5cm
        },
        ritria/.style={
            draw,dashed,shape border uses incircle,
            isosceles triangle,isosceles triangle apex angle=110,
            shape border rotate=-55,yshift=0.1cm
        },
        litria/.style={
            draw,dashed,shape border uses incircle,
            isosceles triangle,isosceles triangle apex angle=110,
            shape border rotate=235,yshift=0.1cm
        }
    }

\lateinabkuerzung{cf}{cf.}
\lateinabkuerzung{Cf}{Cf.}
\lateinabkuerzung{idest}{i.e.}
\lateinabkuerzung{Idest}{I.e.}
\lateinabkuerzung{exempli}{e.g.}
\lateinabkuerzung{Exempli}{E.g.}
\lateinabkuerzung{etcetera}{etc.}
\lateinabkuerzung{etAlia}{et al.}
\lateinabkuerzung{viz}{viz.}
\deutscheabkuerzung{usw}{usw.}
\deutscheabkuerzung{oBdA}{o.\,B.\,d.\,A.}
\deutscheabkuerzung{OBdA}{O.\,B.\,d.\,A.}
\deutscheabkuerzung{oae}{o.\,\"A.}
\deutscheabkuerzung{oE}{o.\,E.}
\deutscheabkuerzung{OE}{O.\,E.}
\deutscheabkuerzung{og}{o.\,g.}
\deutscheabkuerzung{obengen}{o.\,g.}
\deutscheabkuerzung{obenst}{o.\,s.}
\deutscheabkuerzung{untenst}{u.\,s.}
\deutscheabkuerzung{Tfae}{T.\,f.\,a.\,e.}
\deutscheabkuerzung{tfae}{t.\,f.\,a.\,e.}
\deutscheabkuerzung{wrt}{wrt.}
\deutscheabkuerzung{withoutlog}{w.\,l.\,o.\,g.} % !!! WARNING !!! DO NOT define this as wlog. That is reserved.
\deutscheabkuerzung{Withoutlog}{W.\,l.\,o.\,g.} % !!! WARNING !!! DO NOT define this as wlog. That is reserved.
\deutscheabkuerzung{respectively}{resp.}

\def\topInterior#1{\mathop{\textup{int}}(#1)}
\def\oBall#1_#2{\cal{B}_{#2}(#1)}
\def\clBall#1_#2{\quer{\cal{B}}_{#2}(#1)}

\def\Cts{\@ifnextchar_{\Cts@tief}{\Cts@tief_{}}}
    \def\Cts@tief_#1#2{\@ifnextchar\bgroup{\Cts@two_{#1}{#2}}{\Cts@one_{#1}{#2}}}
    \def\Cts@one_#1#2{C_{#1}\big(#2\big)}
    \def\Cts@two_#1#2#3{C_{#1}\big(#2,~#3\big)}

\def\KmpRm#1{\cal{K}(#1)}

\def\Lspace^#1{L^{#1}}

\def\reals{\mathbb{R}}

\def\RealPart{\mathop{\mathfrak{R}\text{\upshape e}}}

\def\Torus{\mathbb{T}}

\def\rtnl{\mathbb{Q}}

\def\onematrix{\text{\upshape\bfseries I}}

\def\zerovector{\text{\upshape\bfseries 0}}

\def\ntrl{\mathbb{N}}
\def\ntrlpos{\mathbb{N}}
\def\ntrlzero{\mathbb{N}_{0}}
\def\realsNonNeg{\reals_{+}}

\def\ntrlpos{\mathbb{N}}
\def\leer{\emptyset}

\def\BRAKET#1#2{\langle{}#1,~#2{}\rangle}
\def\brkt#1{\langle{}#1{}\rangle}

\def\C0{\ensuremath{C_{0}}}
\def\restr#1{\vert_{#1}}

\def\eps{\varepsilon}
\let\altphi\phi
\let\altvarphi\varphi
    \def\phi{\altvarphi}
    \def\varphi{\altphi}

\def\quer#1{\overline{#1}}

\def\lim{\mathop{\ell\mathrm{im}}}

\def\ran{\mathop{\textit{ran}}}

\def\EndBanach{\mathop{\mathfrak{L}}}
\def\BoundedOps#1{\@ifnextchar\bgroup{\BoundedOps@two{#1}}{\EndBanach(#1)}}
    \def\BoundedOps@two#1#2{\EndBanach(#1,#2)}
\def\HilbertRaum{\mathcal{H}}

\def\RaumX{X}
\def\RaumY{Y}

\def\hardSigma02{\cal{Q}}

    \def\OpSpaceU#1{\mathop{\cal{U}}(#1)}
    \def\OpSpaceI#1{\mathop{\cal{I}}(#1)}
    \def\OpSpaceC#1{\mathop{\cal{C}}(#1)}

\def\SpCs{\cal{F}^{c}_{s}}
\def\SpCw{\cal{F}^{c}_{w}}
\def\SpUs{\cal{F}^{u}_{s}}
\def\SpUw{\cal{F}^{u}_{w}}
\def\SpHs{\cal{C}_{s}}
\def\SpHw{\cal{C}_{w}}
\def\SpUHs{\cal{U}_{s}}
\def\SpUHw{\cal{U}_{w}}

    \def\topWOT{\text{\upshape \scshape wot}}
    \def\tinytopWOT{\text{\scriptsize\upshape \scshape wot}}
    
    \def\toplocWOT{\text{{{$\mathpzc{k}$}}}_{\text{\tiny\upshape \scshape wot}}}
    
    \def\tinytoplocWOT{\text{\scriptsize{{{$\mathpzc{k}$}}}-{\upshape \scshape wot}}}
    \def\topSOT{\text{\upshape \scshape sot}}
    \def\tinytopSOT{\text{\scriptsize\upshape \scshape sot}}
    
    \def\toplocSOT{\text{{{$\mathpzc{k}$}}}_{\text{\tiny\upshape \scshape sot}}}
    
    \def\tinytoplocSOT{\text{\scriptsize{{{$\mathpzc{k}$}}}-{\upshape \scshape sot}}}
    \def\topPW{\text{\upshape \scshape pw}}
    \def\tinytopPW{\text{\scriptsize\upshape \scshape pw}}

\def\topLOC{\text{{{$\mathpzc{k}$}}}}
\def\tinytopLOC{\text{\scriptsize{{$\mathpzc{k}$}}}}

\def\minizerlegt{\rotatebox[origin=c]{180}{\ensuremath{\neg}}}
\def\streamaddbasic{\overset{\minizerlegt}{\oplus}} % \obackslash

\def\streamtimesbasic{\overset{\minizerlegt}{\otimes}}

\def\streamadd_#1^#2{\mathbin{{}_{#1}\streamaddbasic_{#2}}}
\def\streamtimes_#1{\mathbin{\underset{#1}{\streamtimesbasic}}}
\@ifundefined{setcitestyle}{%
}{%
    \setcitestyle{numeric-comp,open={[},close={]}}
}

\renewcommand{\arraystretch}{1}
\def\firstparagraph{\noindent}
\def\continueparagraph{\noindent}

\hypersetup{
    hidelinks=true,
}

    \generatenestedsecnumbering{arabic}{section}{subsection}
    \generatenestedsecnumbering{arabic}{subsection}{subsubsection}
    \def\theunitnamesection{\thesection}

    \def\sectionname{}

    \let\appendix@orig\appendix
    \def\appendix{%
        \appendix@orig%
        \let\boolinappendix\boolwahr
        \addcontentsline{toc}{part}{\appendixname}%
        \addtocontents{toc}{\protect\setcounter{tocdepth}{0}}
        \def\sectionname{Appendix}%
        \def\theunitnamesection{\Alph{section}}%
    }
    \def\notappendix{%
        \let\boolinappendix\boolfalse
        \addtocontents{toc}{\protect\setcounter{tocdepth}{1 }}
        \def\sectionname{}%
        \def\theunitnamesection{\arabic{section}}%
    }

    \def\@seccntformat#1{%
        \protect\textup{%
            \protect\@secnumfont
            \expandafter\protect\csname format#1\endcsname%
            \csname the#1\endcsname
            \expandafter\protect\csname format#1@pt\endcsname%
            \space
        }%
    }

    \def\formatsection@text{\centering\Large\scshape}
    \def\formatsection@pt{\secnumberingseppt}
    \def\section{\@startsection{section}{1}{\z@}{.7\linespacing\@plus\linespacing}{.5\linespacing}{\formatsection@text}}

    \def\formatsubsection@text{\flushleft\bfseries\scshape}
    \def\formatsubsection@pt{\subsecnumberingseppt}
    \def\subsection{\@startsection{subsection}{2}{\z@}{\z@}{\z@\hspace{1em}}{\formatsubsection@text}}

\def\rafootnotectr{20}
\def\incrftnotectr#1{%
    \addtocounter{#1}{1}%
    \ifnum\value{#1}>\rafootnotectr\relax
        \setcounter{#1}{0}%
    \fi%
}
\def\footnoteref[#1]{\protected@xdef\@thefnmark{\ref{#1}}\@footnotemark}
\let\altfootnotetext\footnotetext
    \def\footnotetext[#1]#2{\incrftnotectr{footnote}\altfootnotetext[\value{footnote}]{\label{#1}#2}}

    \def\footnotemark[#1]{\text{\textsuperscript{\getrefnumber{#1}}}}

\DefineFNsymbols*{custom}{abcdefghijklmnopqrstuvwxyz}
\setfnsymbol{custom}

\def\kopfzeiledefault{
    \lhead[]{}
    \lhead[]{}
    \chead[]{}
    \rhead[]{}
    \lfoot[]{}
    \cfoot{\footnotesize\thepage}
    \rfoot[]{}
}

\def\aktuellesfont{\csnamermfamily\endcsname}
\def\documentfont{%
    \gdef\aktuellesfont{\csnamermfamily\endcsname}%
    \fontfamily{cmr}\fontseries{m}\selectfont%
    \renewcommand{\sfdefault}{phv}%
    \renewcommand{\ttdefault}{pcr}%
    \renewcommand{\rmdefault}{cmr}% <— funktionieren nicht mit {ptm}
    \renewcommand{\bfdefault}{bx}%
    \renewcommand{\itdefault}{it}%
    \renewcommand{\sldefault}{sl}%
    \renewcommand{\scdefault}{sc}%
    \renewcommand{\updefault}{n}%
}

\allowdisplaybreaks

\def\startdocumentlayoutoptions{
    \selectlanguage{british}
    \setlength{\parskip}{0.25\baselineskip}
    \setlength{\parindent}{2em}
    \kopfzeiledefault
    \documentfont
    \normalsize
}

\def\highlightTerm#1{\emph{#1}}
\newcommand{\highlightForReview}[1]{%
    \bgroup\color{blue}#1\egroup%
}

\def\@adminfootnotes{%
    \let\@makefnmark\relax
    \let\@thefnmark\relax
    \ifx\@empty\@date\else%
        \@footnotetext{\@setdate}%
    \fi%
    \ifx\@empty\@subjclass\else%
        \@footnotetext{\@setsubjclass}%
    \fi
    \ifx\@empty\@keywords\else%
        \@footnotetext{\@setkeywords}%
    \fi
    \ifx\@empty\thankses\else%
        \@footnotetext{\def\par{\let\par\@par}\@setthanks}%
    \fi
}

\def\@settitle{\Large\bfseries\scshape\@title}

\def\@maketitle{%
  \normalfont\normalsize
  \@adminfootnotes
  \@mkboth{\@nx\shortauthors}{\@nx\shorttitle}%
  \global\topskip42\p@\relax
  {\centering\@settitle}
  \ifx\@empty\authors\else{\centering\small\@setauthors}\fi
  \ifx\@empty\@date\else{\vtop{\centering\small\@date\@@par}}\fi
  \ifx\@empty\@dedicatory%
  \else%
    \baselineskip\p@
    \vtop{\centering{\footnotesize\itshape\@dedicatory\@@par}%
    \global\dimen@i\prevdepth}\prevdepth\dimen@i%
  \fi
  \@setabstract
  \normalsize
  \if@titlepage
    \newpage
  \else
    \dimen@34\p@\advance\dimen@-\baselineskip
  \fi
}

\def\addresseshere{%
  \bgroup
  \setlength{\parindent}{0pt}
  \enddoc@text
  \egroup
  \let\enddoc@text\relax
}

\makeatother

\begin{document}
\startdocumentlayoutoptions

%% FRONTMATTER:
\thispagestyle{plain}

\def\abstractname{Abstract}
\begin{abstract}
    Working over infinite dimensional separable Hilbert spaces, residual results have been achieved for
    the space of contractive $\C0$-semigroups
    under the topology of uniform weak operator convergence on compact subsets of $\realsNonNeg$.
    Eisner and Ser\'eny raised in
        \cite{eisnersereny2009catThmStableSemigroups}
        and
        \cite{eisner2010buchStableOpAndSemigroups}
    the open problem:
    Does this space constitute a Baire space?
    Observing that the subspace of unitary semigroups is completely metrisable
    and appealing to known density results,
    we solve this problem positively
    by showing that
    certain topological properties can in general be transferred from dense subspaces to larger spaces.
    The transfer result in turn relies upon classification of topological properties via infinite games.
    Our approach is sufficiently general and can be applied to other contexts,
    \exempli the space of contractions under the $\topPW$-topology.
\end{abstract}

\title[The space of contractive $\C0$-semigroups is a Baire space]{The space of contractive $\C0$-semigroups is a Baire space}
\author{Raj Dahya}
\email{raj.dahya@web.de}
\address{Fakult\"at f\"ur Mathematik und Informatik\newline
Universit\"at Leipzig, Augustusplatz 10, D-04109 Leipzig, Germany}

\maketitle

%% HAUPTTEXT:
\setcounternach{section}{1}

\section[Introduction]{Introduction}
\label{sec:intro}

\firstparagraph
Various residuality results for the space of unitary/isometric/contractive $\C0$-semigroups
over a separable, infinite dimensional Hilbert space,
and viewed with natural topologies,
have been achieved in
    \cite{eisnersereny2009catThmStableSemigroups,krol2009}
    and
    \cite[\S{}III.6 and \S{}IV.3.3]{eisner2010buchStableOpAndSemigroups}.
In a similar vein, residual properties for spaces of operators
(\viz unitaries, isometries, and contractions)
over separable, infinite dimensional Hilbert spaces
and viewed with the weak, strong, and weak-polynomial operator topologies
were investigated in
    \cite{eisner2010typicalContraction,eisnermaitrai2010typicalOperators,Eisner2008categoryThmStableOpHilbert}.

%%%%%%%%%%%%%%%%%%%%%%%%%%%%%%%%%%%%%%%%%%%%%%%%%%%%%%%%%%%%%%%%
%% START OF HISTORICAL KONTEXT
%%

Now, residuality is meaningful provided the larger topological space is at least a Baire space.
The space of contractions under the $\topPW$-topology was proved in
\cite[Theorem~4.1]{eisnermaitrai2010typicalOperators}
to constitute a Polish space and thus a Baire space.
We can actually view the space of contractions as \emph{discrete} contractive semigroups
as shall be made precise later (see \Cref{rem:discete-time-is-pw:sig:article-simpl-wk-problem-raj-dahya}).
The \emph{continuous} case however, \idest the space of $\C0$-semigroups
under the topology of uniform weak operator convergence on compact subsets of $\realsNonNeg$,
remained open.
The history of this is rooted
in the following problem about asymptotic properties of $\C0$-semigroups.
Let $\HilbertRaum$ denote a separable infinite dimensional Hilbert space.

\begin{defn*}
    A $\C0$-semigroup, $T$, on $\HilbertRaum$ is called \highlightTerm{weakly stable}
    if for all $\xi,\eta\in\HilbertRaum$ it holds that
    ${\brkt{T(t)\xi,\eta}\longrightarrow 0}$ as ${t\longrightarrow\infty}$.
    It is called \highlightTerm{almost weakly stable}
    if for all $\xi,\eta\in\HilbertRaum$
    some measurable set, $A\subseteq\realsNonNeg$
    with asymptotic density $1$\footnote{
        \idest{} ${\frac{\lambda(A\cap[0,t])}{\lambda([0,t])}\longrightarrow 1}$
        as ${t\longrightarrow\infty}$,
        where $\lambda$ denotes the Lebesgue measure.
    } exists, such that
    ${\brkt{T(t)\xi,\eta}\longrightarrow 0}$
    as ${t\longrightarrow\infty}$ in $A$.
\end{defn*}

In
    \cite[Theorem~2.5]{eisnersereny2009catThmStableSemigroups}
the following residuality result was proved:

\begin{thm*}[Eisner-Ser\'eny, 2009]
    Consider the spaces of unitary \respectively isometric $\C0$-semigroups on $\HilbertRaum$,
    under the topology of uniform strong operator convergence on compact subsets of $\realsNonNeg$.
    The properties of being
        \highlightTerm{almost weakly stable}
        and
        not \highlightTerm{weakly stable}
    are residual.
\end{thm*}

As we shall show later, the topologies of uniform strong and weak operator convergence
on compact subsets of $\realsNonNeg$ coincide for unitary $\C0$-semigroups
(see \Cref{prop:spHs-and-spUHs-polish:sig:article-simpl-wk-problem-raj-dahya}),
the latter topology being the one of particular interest in this paper.
In
    \cite[\S{}4]{eisnersereny2009catThmStableSemigroups}
the authors remark:
{\itshape It is not clear how to prove an analogue to [the residuality theorems for stability properties] for contractive semigroups.}
It is further remarked, that it be not clear whether the Baire category theorem applies to the space of contractive semigroups
(\cf also \cite[\S{}III.6.3]{eisner2010buchStableOpAndSemigroups}).
In \cite[Corollary~3.2]{krol2009},
Król addressed the first part of this problem.
He showed that the $G_{\delta}$-subspace of unitary $\C0$-semigroups is dense in the space of contractive $\C0$-semigroups,
thereby obtaining that generic contractive $\C0$-semigroups are unitary.
In particular, the above result immediately extends to the contractive case.

Our paper addresses the remaining issue, which completes this history.
Specifically, we positively solve the open problem as to whether the space of contractive $\C0$-semigroups constitutes a Baire space.
Thus we can indeed apply the Baire category theorem to this space and
the above residuality results about asymptotic properties are meaningful for the contractive case.

%%
%% END OF HISTORICAL KONTEXT
%%%%%%%%%%%%%%%%%%%%%%%%%%%%%%%%%%%%%%%%%%%%%%%%%%%%%%%%%%%%%%%%

%%%%%%%%%%%%%%%%%%%%%%%%%%%%%%%%%%%%%%%%%%%%%%%%%%%%%%%%%%%%%%%%
%% START OF HISTORICAL U IS RESIDUAL IN C + RIGIDITY
%%

Now, the knowledge that this space is a Baire space
yields an analogous observation to \cite[Theorem~2.2]{eisner2010typicalContraction} about the single operator case,
\viz that unitary $\C0$-semigroups are residual in the (Baire) space of contractive $\C0$-semigroups
(see \Cref{cor:unitary-C0-residual-in-contractive-C0:sig:article-simpl-wk-problem-raj-dahya} below).
And building on this provides an immediate result in the study of rigidity phenomena
(which we prove in \S{}\ref{sec:applications}):

\begin{thm}
\makelabel{thm:application-rigidity:sig:article-simpl-wk-problem-raj-dahya}
    Let $\HilbertRaum$ be a separable infinite dimensional space.
    Then the set of all contractive $\C0$-semigroups, $T$, on $\HilbertRaum$,
    satisfying the following properties

    \begin{kompaktenum}{\bfseries (i)}[\rtab]
        \item\punktlabel{1}
            There exists a set $A\subseteq\realsNonNeg$ of density $1$,%
            \footnote{
                \idest{} ${\frac{\lambda(A\cap[0,t])}{\lambda([0,t])}\longrightarrow 1}$
                as ${t\longrightarrow\infty}$,
                where $\lambda$ denotes the Lebesgue measure.
            }
            such that ${(T(t))_{t\in A}\overset{\tinytopWOT}{\longrightarrow}0}$.
        \item\punktlabel{2}
            For each ${\alpha\in\Torus}$ (the unit circle in the complex plane)
            a subnet of ${(T(t))_{t\in\realsNonNeg}}$
            exists that converges asymptotically to $\alpha\onematrix$ under the $\topSOT$-topology.
    \end{kompaktenum}

    \continueparagraph
    is residual in the space of contractive $\C0$-semigroups,
    endowed with the topology of uniform weak operator convergence
    on compact subsets of $\realsNonNeg$.
\end{thm}

This solves the open problem raised in
    \cite[Remark~IV.3.23]{eisner2010buchStableOpAndSemigroups}.
In this book, the above was proved in the discrete setting
    (see \cite[Theorem~IV.3.11]{eisner2010buchStableOpAndSemigroups})
as well for the unitary and isometric case in the continuous setting
    (see \cite[Theorem~IV.3.20]{eisner2010buchStableOpAndSemigroups}).

%%
%% END OF HISTORICAL U IS RESIDUAL IN C + RIGIDITY
%%%%%%%%%%%%%%%%%%%%%%%%%%%%%%%%%%%%%%%%%%%%%%%%%%%%%%%%%%%%%%%%

Now, our approach to showing the Baire space property is a uniform one.
It additionally provides an alternative proof
that the space of contractions under the $\topPW$-topology is a Baire space
(\cf \cite[Theorem~4.1]{eisnermaitrai2010typicalOperators}).

Finally note that residuality properties of (contractive) operators on Banach spaces,
as initiated in \cite{eisner2010typicalContraction,eisnermaitrai2010typicalOperators}, have further recently been studied in connection with
\emph{hypercyclicity} and
the \emph{Invariant Subspace Problem}
in \cite{grivaux2017linear,grivaux2020does}. The continuous case has not yet been investigated.

\subsection[Spaces of operators and operator-valued functions]{Spaces of operators and operator-valued functions}
\label{sec:intro:spaces}

\def\compactcover{\tilde{\cal{K}}}
\def\spaceCoher{\mathop{\textup{coher}}}

\firstparagraph
Before formulating our main result, we need to define the key topological spaces.
Throughout, fix a separable, infinite dimensional Hilbert space, $\HilbertRaum$.
We denote via

    \begin{mathe}[mc]{rcccccl}
        \BoundedOps{\HilbertRaum}
        &\supseteq
            &\OpSpaceC{\HilbertRaum}
        &\supseteq
            &\OpSpaceI{\HilbertRaum}
        &\supseteq
            &\OpSpaceU{\HilbertRaum}\\
    \end{mathe}

\continueparagraph
(from left to right) the spaces of bounded linear operators, contractions, isometries and unitaries over $\HilbertRaum$.
These can be endowed with the weak operator topology ($\topWOT$),
the strong operator topology ($\topSOT$),
and also the \highlightTerm{weak polynomial topology} ($\topPW$),
which is defined via the convergence condition

    \begin{mathe}[mc]{rcl}
        T_{i} \overset{\tinytopPW}{\longrightarrow} T
            &:\Longleftrightarrow
                &\forall{n\in\ntrlpos:~}T_{i}^{n}\overset{\tinytopWOT}{\longrightarrow}T^{n}\\
    \end{mathe}

\continueparagraph
for all nets ${(T_{i})_{i}\subseteq\BoundedOps{\HilbertRaum}}$ and all ${T\in\BoundedOps{\HilbertRaum}}$.

We first note the following well-known basic results about the complete metrisability
of the operator spaces
(\cf
    \cite[Exercise~3.4~(5), Exercise~4.9, and Examples~9.B~(6)]{kech1994}
    and
    \cite[Lemma~2.1]{eisner2010typicalContraction}%
):

\begin{prop}
\makelabel{prop:space-of-contractions-Poln:sig:article-simpl-wk-problem-raj-dahya}
    Let $A$ be either $\OpSpaceC{\HilbertRaum}$, $\OpSpaceI{\HilbertRaum}$ or $\OpSpaceU{\HilbertRaum}$.
    Then $(A,\topSOT)$ and $(A,\topWOT)$ are Polish.
    Moreover, the $\topSOT$- and $\topWOT$-topologies
    coincide in the case of
        $\OpSpaceI{\HilbertRaum}$ and $\OpSpaceU{\HilbertRaum}$
    respectively.
\end{prop}

In addition to spaces of operators,
we shall also consider the spaces of unitary and contractive $\C0$-semigroups.
These can be viewed as subspaces of
    $\Cts{\realsNonNeg}{A}$, where $A$ can be $\OpSpaceU{\HilbertRaum}$ or $\OpSpaceC{\HilbertRaum}$,
topologised appropriately.
These spaces of operator-valued functions shall be denoted and topologised as follows:

\begin{defn}
\makelabel{defn:standard-funct-spaces:sig:article-simpl-wk-problem-raj-dahya}
    Let
        ${\SpCs(\realsNonNeg) \colonequals \Cts{\realsNonNeg}{(\OpSpaceC{\HilbertRaum},\topSOT)}}$
    and
        ${\SpUs(\realsNonNeg) \colonequals \Cts{\realsNonNeg}{(\OpSpaceU{\HilbertRaum},\topSOT)}}$,
    \idest these denote the spaces of $\topSOT$-continuous contraction- \respectively unitary-valued functions defined on $\realsNonNeg$.
    Further denote via
        $\SpHs(\realsNonNeg)$ and $\SpUHs(\realsNonNeg)$
        the subspaces of
        $\topSOT$-continuous contractive \respectively unitary semigroups over $\HilbertRaum$.
    Let $\SpCw(\realsNonNeg)$, $\SpUw(\realsNonNeg)$, $\SpHw(\realsNonNeg)$, $\SpUHw(\realsNonNeg)$ denote the $\topWOT$-continuous counterparts.
\end{defn}

\begin{rem}
\makelabel{rem:standard-funct-spaces:arbitrary-top-monoids:sig:article-simpl-wk-problem-raj-dahya}
    Note that we may similarly define function spaces
        $\SpCs(\RaumX)$ and $\SpUs(\RaumX)$
    for any topological space $\RaumX$, not just $\realsNonNeg$.
    We may also define the subspaces of continuous semigroups for any topological monoid,
        $(M,\cdot,1)$, not just $(\realsNonNeg,+,0)$.
    Here, we call any operator-valued function ${T:M\to\BoundedOps{\HilbertRaum}}$
    a semigroup over $\HilbertRaum$ on $(M,\cdot,1)$,
    if it satisfies
        $T(1)=\onematrix$ and $T(s\cdot t)=T(s)T(t)$ for all $s,t\in M$.
    For example, for $(\realsNonNeg^{d},+,\zerovector)$
    the semigroups are referred to as multiparameter semigroups
    and in the case of $(\ntrlzero,+,0)$,
    we may speak of discrete-time semigroups.
\end{rem}

The set of all continuous-valued operator-valued functions is a sufficiently broad but natural context in which to work.
It suffices to topologise the full function spaces and simply endow the subspace of semigroups with the relative topology.
We can additionally detach ourselves for the moment from the space of time points, $\realsNonNeg$,
and consider any locally compact Polish space in its stead.
And, taking notice of \Cref{prop:space-of-contractions-Poln:sig:article-simpl-wk-problem-raj-dahya},
we may also replace the space of values by any Polish space.
This motivates usage of the following general topological definitions.

\begin{defn}
    For any topological space, $\RaumX$, we denote with $\KmpRm{\RaumX}$
    the collection of all compact subsets of $\RaumX$.
\end{defn}

\begin{defn}
    \makelabel{defn:func-spaces-top:u:abstract:sig:article-simpl-wk-problem-raj-dahya}
    Let $\RaumX$ be any \emph{compact} space and $(\RaumY,d)$ be a metric space.
    On the space of continuous functions, $\Cts{\RaumX}{\RaumY}$,
    the topology of
        \highlightTerm{uniform convergence}
    is defined via the convergence condition

    \begin{mathe}[mc]{rcl}
        f^{(i)} \overset{u}{\underset{i}{\longrightarrow}} f
            &:\Longleftrightarrow
                &\displaysup_{t\in\RaumX}d(f^{(i)}(t),f(t))\underset{i}{\longrightarrow}0\\
    \end{mathe}

    \continueparagraph
    for all nets $(f^{(i)})_{i}\subseteq\Cts{\RaumX}{\RaumY}$
    and all $f\in\Cts{\RaumX}{\RaumY}$.
    We denote this space as $(\Cts{\RaumX}{\RaumY},u)$.
\end{defn}

\begin{rem}
    \makelabel{rem:u-indep-metric:sig:article-simpl-wk-problem-raj-dahya}
    If $\RaumX$ is compact, the topology of uniform convergence is
    independent of the choice of metric on $\RaumY$.
    Moreover, even if $\RaumY$ is completely metrisable,
    the choice of a compatible metric, $d$, in the above definition
    does not even have to be complete.
    See \exempli \cite[Lemma~3.98]{aliprantis2005} for a proof of this.
\end{rem}

\begin{defn}[$\topLOC$-topology]
    \makelabel{defn:func-spaces-top:loc:abstract:sig:article-simpl-wk-problem-raj-dahya}
    Let $\RaumX$ be any topological space and $(\RaumY,d)$ be a metric space.
    On the space of continuous functions, $\Cts{\RaumX}{\RaumY}$,
    we define the topology of
        \highlightTerm{uniform convergence on compact subsets of $\RaumX$}
    as follows

    \begin{mathe}[mc]{rcll}
        f^{(i)} \overset{\tinytopLOC}{\underset{i}{\longrightarrow}} f
            &:\Longleftrightarrow
                &\forall{K\in\KmpRm{\RaumX}:~}
                    f^{(i)}\restr{K}\overset{u}{\underset{i}{\longrightarrow}}f\restr{K}\\
    \end{mathe}

    \continueparagraph
    for all nets $(f^{(i)})_{i}\subseteq\Cts{\RaumX}{\RaumY}$
    and all $f\in\Cts{\RaumX}{\RaumY}$.
    We refer to this as the $\topLOC$-topology.
\end{defn}

\begin{rem}
    \makelabel{rem:loc-indep-metric:sig:article-simpl-wk-problem-raj-dahya}
    Clearly, if $\RaumX$ itself is compact, then the topologies on
        $(\Cts{\RaumX}{\RaumY},\topLOC)$
        and $(\Cts{\RaumX}{\RaumY},u)$
    coincide, so we can view the $\topLOC$-topology
    as a generalisation of the uniform topology.
    Furthermore, as per \Cref{rem:u-indep-metric:sig:article-simpl-wk-problem-raj-dahya},
    since the topology on each space
        $(\Cts{K}{\RaumY},u)$ for $K\subseteq\RaumX$ compact
    is independent of the choice of compatible metric on $\RaumY$,
    the $\topLOC$-topology is clearly independent of the choice of metric on $\RaumY$.
\end{rem}

\begin{rem}
    \makelabel{rem:loc-top-equals-kompakt-open-top:sig:article-simpl-wk-problem-raj-dahya}
    Recall that for arbitrary (not necessarily metrisable) topological spaces,
    the \highlightTerm{compact-open topology} on $\Cts{\RaumX}{\RaumY}$
    is generated by a subbasis of open sets of the form
        $\{f\in\Cts{\RaumX}{\RaumY} \mid f(K)\subseteq U\}$
    for $K\in\KmpRm{\RaumX}$ and $U\subseteq\RaumY$ open.
    If $\RaumY$ is now assumed to be metrisable,
    then the compact-open topology
    is equivalent to the $\topLOC$-topology as per \Cref{defn:func-spaces-top:loc:abstract:sig:article-simpl-wk-problem-raj-dahya}.
    To see this, fix a compatible metric, $d$, on $\RaumY$.
    Relying on the definitions, one can see that the sets
        $\{g\in\Cts{\RaumX}{\RaumY} \mid \sup_{t \in K}d(g(t),f(t)) < \eps\}$
        for $f \in \Cts{\RaumX}{\RaumY}$, $K\in\KmpRm{\RaumX}$, and $\eps > 0$,
    constitute a basis for the $\topLOC$-topology.
    Observe (i) that for each
        $f\in\Cts{\RaumX}{\RaumY}$,
        $K\in\KmpRm{\RaumX}$,
        and
        $\eps>0$,
    from the compactness of $K$ and continuity of $f$
    one can find a finite open cover
        $\cal{O}$ of $K$
        and
        $t_{W}\in K\cap W$ for each $W\in\cal{O}$,
    such that $\sup_{t\in K\cap W}d(f(t),f(t_{W}))<\eps/2$
    for all $W\in\cal{O}$.
    It follows that for each $g\in\Cts{\RaumX}{\RaumY}$,
    if
        $g(K\cap\quer{W})\subseteq\oBall{f(t_{W})}_{\eps/2}$ for each $W\in\cal{O}$
    then $\sup_{t\in K}d(g(t),f(t))<\eps$.
    Conversely, observe (ii) that for each
        $f\in\Cts{\RaumX}{\RaumY}$,
        $n\in\ntrlpos$,
        $K_{1},K_{2},\ldots,K_{n}\in\KmpRm{\RaumX}$,
        and
        $U_{1},U_{2},\ldots,U_{n}\subseteq\RaumY$ open,
    such that
        $f(K_{i})\subseteq U_{i}$ for each $i$,
    from the compactness of the $K_{i}$ and hence of the continuous images, $f(K_{i})$,
    one can find some
        $\eps>0$
    such that
        $\oBall{f(K_{i})}_{\eps}\subseteq U_{i}$ for each $i$.
    It follows that for each $g\in\Cts{\RaumX}{\RaumY}$,
    if
        $\sup_{t\in\bigcup_{i=1}^{n}K_{i}}d(g(t),f(t))<\eps$,
    then $g(K_{i})\subseteq U_{i}$ for each $i$.
    Observations (i) and (ii) imply that
        the $\topLOC$-topology
        and
        the compact-open topology
    coincide on $\Cts{\RaumX}{\RaumY}$,
    when $\RaumY$ is metrisable.
\end{rem}

This general definition allows us to readily demonstrate complete metrisability.

\begin{prop}
\makelabel{prop:loc-abstract-basic:sig:article-simpl-wk-problem-raj-dahya}
    Let $\RaumX$ be a locally compact Polish space
    and $\RaumY$ be Polish.
    Then $(\Cts{\RaumX}{\RaumY},\topLOC)$ is Polish.
\end{prop}

    \begin{proof}
        Since $\RaumX$ is locally compact and Polish,
        one can find a countable collection
            $\compactcover\subseteq\KmpRm{\RaumX}$
        such that
            $\{\topInterior{K}\mid K\in\compactcover\}$ is an open cover of $\RaumX$
        (%
            \exempli
            in the case of $\RaumX=\realsNonNeg$, one may set $\compactcover \colonequals \{[0,n]\mid n\in\ntrl\}$,
            and in the case of $\RaumX=\ntrlzero$, one can set $\compactcover \colonequals \{\{0,1,\ldots,n\}\mid n\in\ntrlzero\}$%
        ).
        Letting $(K_{n})_{n\in\ntrl}$ be an enumeration of $\compactcover$,
        we construct the map

        \begin{mathe}[mc]{rcccl}
            \Psi &: &\Cts{\RaumX}{\RaumY}
                    &\to &\prod_{n\in\ntrl}\Cts{K_{n}}{\RaumY}\\
                &&f &\mapsto &(f\restr{K_{n}})_{n\in\ntrl},\\
        \end{mathe}

        \continueparagraph
        where the right hand space is endowed with the product topology
        and each of the factors is endowed with the uniform topology.
        Exploiting
            the fact that the interiors of the compact sets in $\compactcover$ cover $\RaumX$
            and that $\RaumX$ is locally compact,
        it is straightforward to see that
        this defines an injective, bi-continuous map,
        whose image is the subspace of coherent sequences,

        \begin{mathe}[mc]{rcccl}
            \spaceCoher(\compactcover) &\colonequals &\{
                (a_{n})_{K_{n}} \in \prod_{n\in\ntrl}\Cts{K_{n}}{\RaumY}
                \mid
                \forall{m,n\in\ntrl:~}
                    a_{m}\restr{K_{m}\cap K_{n}}
                    =a_{n}\restr{K_{m}\cap K_{n}}
            \},\\
        \end{mathe}

        \continueparagraph
        which in turn is a closed subspace.
        Since $\RaumX$ and $\RaumY$ are Polish,
        each of the factor spaces, $\Cts{K_{n}}{\RaumY}$, are Polish
        (see \exempli \cite[Theorem~4.19]{kech1994} or \cite[Lemma~3.96--7,~3.99]{aliprantis2005}).
        And since the class of Polish spaces is closed under countable products
        (see \exempli \cite[Corollary~3.39]{aliprantis2005}), it follows that the product space and thus also the closed subspace,
        ${\ran(\Psi)=\spaceCoher(\compactcover)}$, are Polish.
        Since $\Psi$ is a topological embedding,
        it follows that $(\Cts{\RaumX}{\RaumY},\topLOC)$ itself is Polish.
    \end{proof}

Now consider the space, $\RaumY=A\subseteq\OpSpaceC{\HilbertRaum}$,
of unitaries or contractions over $\HilbertRaum$
under the $\topSOT$- and $\topWOT$-topologies.
Since by \Cref{prop:space-of-contractions-Poln:sig:article-simpl-wk-problem-raj-dahya}
these spaces are Polish,
we can apply \Cref{prop:loc-abstract-basic:sig:article-simpl-wk-problem-raj-dahya}
to see that
    $(\Cts{\RaumX}{(A,\topSOT)},\topLOC)$
and $(\Cts{\RaumX}{(A,\topWOT)},\topLOC)$
are Polish.

As per \Cref{rem:loc-top-equals-kompakt-open-top:sig:article-simpl-wk-problem-raj-dahya}
the general $\topLOC$-topology coincides with the compact-open topology.
In a very similar way as was argued in that remark,
it is a straightforward matter to observe
that the compact-open topologies on
    $\Cts{\RaumX}{(A,\topSOT)}$ and $\Cts{\RaumX}{(A,\topWOT)}$
may be equivalently presented as the $\toplocSOT$- and $\toplocWOT$-topologies respectively,
which are defined as follows:

\begin{defn}[$\toplocSOT$-Topology]
    \makelabel{defn:func-spaces-top:loc-sot:abstract:sig:article-simpl-wk-problem-raj-dahya}
    Let $\RaumX$ be any topological space.
    Let $A\subseteq\OpSpaceC{\HilbertRaum}$ be endowed with the $\topSOT$-topology.
    On the space of \topSOT-continuous contraction-valued functions, $\Cts{\RaumX}{A}$,
    the topology of
        \highlightTerm{uniform \topSOT-convergence on compact subsets of $\RaumX$}
    is defined via the convergence condition

    \begin{mathe}[mc]{rcl}
        T_{i} \overset{\tinytoplocSOT}{\underset{i}{\longrightarrow}} T
            &:\Longleftrightarrow
                &\forall{\xi\in\HilbertRaum:~}
                \forall{K\in\KmpRm{\RaumX}:~}
                \displaysup_{t\in K}\|(T_{i}(t)-T(t))\xi\|\underset{i}{\longrightarrow}0\\
    \end{mathe}

    \continueparagraph
    for all nets $(T_{i})_{i}\subseteq\Cts{\RaumX}{A}$
    and all $T\in\Cts{\RaumX}{A}$.
    We refer to this as the $\toplocSOT$-topology.
\end{defn}

\begin{defn}[$\toplocWOT$-Topology]
    \makelabel{defn:func-spaces-top:loc-wot:abstract:sig:article-simpl-wk-problem-raj-dahya}
    Let $\RaumX$ be any topological space.
    Let $A\subseteq\OpSpaceC{\HilbertRaum}$ be endowed with the $\topWOT$-topology.
    On the space of \topWOT-continuous contraction-valued functions, $\Cts{\RaumX}{A}$,
    the topology of
        \highlightTerm{uniform \topWOT-convergence on compact subsets of $\RaumX$}
    is defined via the convergence condition

    \begin{mathe}[mc]{rcl}
        T_{i} \overset{\tinytoplocWOT}{\underset{i}{\longrightarrow}} T
            &:\Longleftrightarrow
                &\forall{\xi,\eta\in\HilbertRaum:~}
                \forall{K\in\KmpRm{\RaumX}:~}
                \displaysup_{t\in K}|\BRAKET{(T_{i}(t)-T(t))\xi}{\eta}|\underset{i}{\longrightarrow}0\\
    \end{mathe}

    \continueparagraph
    for all nets $(T_{i})_{i}\subseteq\Cts{\RaumX}{A}$
    and all $T\in\Cts{\RaumX}{A}$.
    We refer to this as the $\toplocWOT$-topology.
\end{defn}

These definitions and \Cref{prop:loc-abstract-basic:sig:article-simpl-wk-problem-raj-dahya}
allow us obtain the following basic complete metrisability results
for the spaces defined in \Cref{defn:standard-funct-spaces:sig:article-simpl-wk-problem-raj-dahya}.

\begin{prop}[Complete metrisability of operator-valued function spaces]
    \makelabel{prop:op-func-spaces-are-polish:basic-func:sig:article-simpl-wk-problem-raj-dahya}
    Let $\RaumX$ be a locally compact Polish space,
    \exempli $\RaumX\in\{\realsNonNeg^{d},\ntrlzero^{d}\mid d\in\ntrlpos\}$.
    Then
        $(\SpCs(\RaumX),\toplocSOT)$,
        $(\SpUs(\RaumX),\toplocSOT)$,
        $(\SpCw(\RaumX),\toplocWOT)$, and
        $(\SpUw(\RaumX),\toplocWOT)$,
    are Polish spaces.
    Furthermore,
        $(\SpUs(\RaumX),\toplocSOT)=(\SpUw(\RaumX),\toplocWOT)$,
    \idest, these spaces coincide in terms of their elements and their topologies.
\end{prop}

    \begin{proof}
        Unpacking the notation in \Cref{defn:standard-funct-spaces:sig:article-simpl-wk-problem-raj-dahya}
        we need to show that
            $(\Cts{\RaumX}{(A,\topSOT)},\toplocSOT)$
            and
            $(\Cts{\RaumX}{(A,\topWOT)},\toplocWOT)$
        are Polish for
            $A\in\{\OpSpaceC{\HilbertRaum},\OpSpaceU{\HilbertRaum}\}$.
        Fixing $A$, we note by
            \Cref{%
                prop:space-of-contractions-Poln:sig:article-simpl-wk-problem-raj-dahya,%
                prop:loc-abstract-basic:sig:article-simpl-wk-problem-raj-dahya%
            }
        that
            $(\Cts{\RaumX}{(A,\topSOT)},\topLOC)$
        and
            $(\Cts{\RaumX}{(A,\topWOT)},\topLOC)$
        are Polish.
        As per the comments before
            \Cref{defn:func-spaces-top:loc-sot:abstract:sig:article-simpl-wk-problem-raj-dahya},
        these spaces are equal to
            $(\Cts{\RaumX}{(A,\topSOT)},\toplocSOT)$
        and $(\Cts{\RaumX}{(A,\topWOT)},\toplocWOT)$
        respectively.
        So the first claims hold.

        For the final claim, let $A=\OpSpaceU{\HilbertRaum}$.
        By \Cref{prop:space-of-contractions-Poln:sig:article-simpl-wk-problem-raj-dahya} $(A,\topSOT)=(A,\topWOT)$.
        Thus
            $(\Cts{\RaumX}{(A,\topSOT)},\topLOC)
            =(\Cts{\RaumX}{(A,\topWOT)},\topLOC)$,
        since the $\topLOC$-topology depends only on the topologies of the underlying spaces
        (\cf \Cref{rem:loc-indep-metric:sig:article-simpl-wk-problem-raj-dahya}).
        As above, these spaces are equal to
            $(\Cts{\RaumX}{(A,\topSOT)},\toplocSOT)$
        and $(\Cts{\RaumX}{(A,\topWOT)},\toplocWOT)$
        respectively.
        Hence $(\SpUs(\RaumX),\toplocSOT)=(\SpUw(\RaumX),\toplocWOT)$.
    \end{proof}

\begin{rem}
\makelabel{rem:direct-explanation-equality-of-topologies:sig:article-simpl-wk-problem-raj-dahya}
    Note that we relied on the general definition of the $\topLOC$-topology
    in order to demonstrate complete metrisability.
    We also drew out its usage here in order to argue that
    the $\toplocSOT$- and $\toplocWOT$-topologies coincide
    on ${\SpUs(\RaumX)=\SpUw(\RaumX)}$.
    For the reader's convenience we present a direct proof of the latter
    as an alternative:

    Consider a net of $\topSOT$-continuous ($\equiv$ $\topWOT$-continuous)
    unitary-valued functions ${(T^{(i)})_{i}\subseteq\SpUs(\RaumX)=\SpUw(\RaumX)}$
    and ${T\in\SpUs(\RaumX)=\SpUw(\RaumX)}$.
    Clearly, if
        ${T^{(i)}\overset{\tinytoplocSOT}{\underset{i}{\longrightarrow}}T}$,
    then ${T^{(i)}\overset{\tinytoplocWOT}{\underset{i}{\longrightarrow}}T}$.
    For the converse, assume the latter convergence holds.
    Let $K\in\KmpRm{\RaumX}$ and $\xi\in\HilbertRaum$ be arbitrary.
    We need to show

        \begin{mathe}[mc]{rcl}
            \eqtag[eq:1:rem:direct-explanation-equality-of-topologies:sig:article-simpl-wk-problem-raj-dahya]
            \displaysup_{t \in K}\|(T^{(i)}(t)-T(t))\xi\| &\underset{i}{\longrightarrow} &0.\\
        \end{mathe}

    \continueparagraph
    To achieve this, first consider an arbitrary fixed $\eps>0$.
    One may take advantage of the fact that $T$ is $\topSOT$-continuous,
    and thus that $T(\cdot)\xi$ is norm-continuous on the compact subset $K\subseteq\RaumX$,
    to obtain a finite open cover, $\cal{O}$,
    of $K$ and $t_{W}\in W$ for each $W\in\cal{O}$,
    such that

        \begin{mathe}[mc]{c}
            \eqtag[eq:2:rem:direct-explanation-equality-of-topologies:sig:article-simpl-wk-problem-raj-dahya]
            \displaysup_{t\in W\cap K}\|(T(t)-T(t_{W}))\xi\| < \eps
        \end{mathe}

    \continueparagraph
    holds for each $W\in\cal{O}$.
    Since $\cal{O}$ is a cover of $K$, we thus have

    \begin{mathe}[mc]{rcl}
        \eqtag[eq:3:rem:direct-explanation-equality-of-topologies:sig:article-simpl-wk-problem-raj-dahya]
            \displaysup_{t\in K}\|(T^{(i)}(t)-T(t))\xi\|
            &= &\displaymax_{W\in\cal{O}}
                \displaysup_{t\in W\cap K}
                    \|(T^{(i)}(t)-T(t))\xi\|\\
            &\leq &\displaymax_{W\in\cal{O}}
                \displaysup_{t\in W\cap K}
                    \begin{array}[t]{0l}
                        \left(\|(T(t)-T(t_{W}))\xi\|\right.\\
                        + \left.\|(T^{(i)}(t)-T(t_{W}))\xi\|\right)\\
                    \end{array}\\
            &\eqcrefoverset{eq:2:rem:direct-explanation-equality-of-topologies:sig:article-simpl-wk-problem-raj-dahya}{%
                \leq%
            } &\eps +
                \displaymax_{W\in\cal{O}}
                \displaysup_{t\in W\cap K}
                    \|(T^{(i)}(t)-T(t_{W}))\xi\|\\
    \end{mathe}

    \continueparagraph
    for each index $i$.
    Noting also that
        $T^{(i)}(t)$ and $T(t)$
    are isometries for all $t\in\RaumX$ yields

    \begin{mathe}[mc]{rcl}
        \eqtag[eq:4:rem:direct-explanation-equality-of-topologies:sig:article-simpl-wk-problem-raj-dahya]
        \|(T^{(i)}(t)-T(t_{W}))\xi\|^{2}
            &= &\|T^{(i)}(t)\xi\|^{2}
                + \|T(t_{W})\xi\|^{2}
                - 2\RealPart\BRAKET{T^{(i)}(t)\xi}{T(t_{W})\xi}\\
            &= &\begin{array}[t]{0l}
                \|T(t)\xi\|^{2} + \|T(t_{W})\xi\|^{2}
                    - 2\RealPart\BRAKET{T(t)\xi}{T(t_{W})\xi}\\
                    - 2\RealPart\BRAKET{(T^{(i)}(t)-T(t))\xi}{T(t_{W})\xi}\\
                \end{array}\\
            &= &\begin{array}[t]{0l}
                \|(T(t)-T(t_{W}))\xi\|^{2}
                    - 2\RealPart\BRAKET{(T^{(i)}(t)-T(t))\xi}{T(t_{W})\xi}\\
                \end{array}\\
            &\eqcrefoverset{eq:2:rem:direct-explanation-equality-of-topologies:sig:article-simpl-wk-problem-raj-dahya}{%
                \leq%
            } &\eps^{2}
               + 2|\BRAKET{(T^{(i)}(t)-T(t))\xi}{\eta_{W}}|\\
    \end{mathe}

    \continueparagraph
    for all $t\in W\cap K$, all $W\in\cal{O}$, and all indexes $i$,
    where $\eta_{W} \colonequals T(t_{W})\xi$.
    Combining \eqcref{eq:4:rem:direct-explanation-equality-of-topologies:sig:article-simpl-wk-problem-raj-dahya}
    and \eqcref{eq:3:rem:direct-explanation-equality-of-topologies:sig:article-simpl-wk-problem-raj-dahya}
    yields

        \begin{mathe}[mc]{rcl}
            \displaysup_{t\in K}\|(T^{(i)}(t)-T(t))\xi\|
                &\leq
                    &\eps
                    + \displaymax_{W\in\cal{O}}
                        \sqrt{%
                            \eps^{2}
                            + 2\cdot\displaysup_{t\in W\cap K}|\BRAKET{(T(t)-T^{(i)}(t))\xi}{\eta_{W}}|
                        }\\
                &\leq
                    &\eps
                    + \displaymax_{W\in\cal{O}}
                        \sqrt{%
                            \eps^{2}
                            + 2\cdot\displaysup_{t\in K}|\BRAKET{(T(t)-T^{(i)}(t))\xi}{\eta_{W}}|
                        }\\
        \end{mathe}

    \continueparagraph
    for all indexes $i$.
    Since the cover, $\cal{O}$, is finite and ${T^{(i)}\overset{\tinytoplocWOT}{\underset{i}{\longrightarrow}}T}$,
    the right hand expression clearly converges to $2\eps$.
    Since $\eps>0$ was arbitrarily chosen,
    it follows that the limit in
        \eqcref{eq:1:rem:direct-explanation-equality-of-topologies:sig:article-simpl-wk-problem-raj-dahya}
    holds.%
\end{rem}

Recalling that the $\topSOT$- and $\topWOT$-topologies also coincide on the space of isometries
(\cf \Cref{prop:space-of-contractions-Poln:sig:article-simpl-wk-problem-raj-dahya}),
and since the above argument only relied on the operator-valued functions being isometry-valued,
the claims in \Cref{prop:op-func-spaces-are-polish:basic-func:sig:article-simpl-wk-problem-raj-dahya}
remain true if we consider spaces of isometrie-valued functions instead of unitary-valued functions.

\subsection[Statement of the main result]{Statement of the main result}
\label{sec:intro:preview}

\firstparagraph
We may now apply the definitions of the previous section
to the subspaces of unitary and contractive semigroups.

\begin{prop}
\makelabel{prop:spHs-and-spUHs-polish:sig:article-simpl-wk-problem-raj-dahya}
    Let $(M,\cdot,1)$ be a locally compact Polish monoid
        (\exempli $(\realsNonNeg,+,0)$ or $(\ntrlzero,+,0)$).
    Then
        $(\SpHs(M),\toplocSOT)$ and $(\SpUHs(M),\toplocSOT)$
    are Polish spaces.
    Moreover, ${(\SpUHs(M),\toplocSOT)=(\SpUHw(M),\toplocWOT)}$.
\end{prop}

    \begin{proof}
        For any ${T:M\to\BoundedOps{\HilbertRaum}}$,
        let $\Phi(T)$ denote that $T$ has the semigroup property,
        \idest $T(1)=\onematrix$ and $T(s\cdot t)=T(s)T(t)$
        for all $s,t\in M$.
        By definition (see \Cref{rem:standard-funct-spaces:arbitrary-top-monoids:sig:article-simpl-wk-problem-raj-dahya})
        we have
            $\SpUHs(M)=\{T\in\SpUs(M)\mid\Phi(T)\}$
        and $\SpUHw(M)=\{T\in\SpUw(M)\mid\Phi(T)\}$.
        Thus, applying \Cref{prop:op-func-spaces-are-polish:basic-func:sig:article-simpl-wk-problem-raj-dahya}
        to these subspaces immediately yields the final claim.

        Since by
            \Cref{prop:op-func-spaces-are-polish:basic-func:sig:article-simpl-wk-problem-raj-dahya}
            $(\SpCs(M),\toplocSOT)$ and $(\SpUs(M),\toplocSOT)$ are Polish,
        and since
            ${\SpUHs(M)=\SpHs(M)\cap\SpUs(M)}$
        and
            $(\SpUs(M),\toplocSOT)$ can be viewed as a subspace of $(\SpCs(M),\toplocSOT)$,
        in order to prove the first claims,
        it clearly suffices to prove that
            $\SpHs(M)$ is closed within $(\SpCs(M),\toplocSOT)$.
        To this end, consider an arbitrary net, ${(T_{i})_{i}\subseteq\SpHs(M)}$,
        with ${T_{i}\overset{\tinytoplocSOT}{\longrightarrow}T}$ for some ${T\in\SpCs(M)}$.
        We need to show that $T\in\SpHs(M)$.
        Now, for each $t\in M$, since $\{t\}\subseteq M$ is compact,
        the definition of uniform convergence on compact subsets
        immediately yields that ${T_{i}(t)\overset{\tinytopSOT}{\longrightarrow}T(t)}$.
        Hence $T(1)=\lim_{i}T_{i}(1)=\onematrix$.
        Since operator-multiplication is $\topSOT$-continuous on norm-bounded subsets,
        and since $T_{i}$ and $T$ are uniformly bounded,
        it follows that
            ${T_{i}(s \cdot t)=T_{i}(s)T_{i}(t)\overset{\tinytopSOT}{\longrightarrow}T(s)T(t)}$
        and
            ${T_{i}(s \cdot t)\overset{\tinytopSOT}{\longrightarrow}T(s \cdot t)}$,
        and hence $T(s \cdot t)=T(s)T(t)$ for all $s,t\in M$.
        Thus $T\in\SpHs(M)$, which proves that $\SpHs(M)$ is a $\toplocSOT$-closed subset of $\SpCs(M)$.
    \end{proof}

\begin{rem}
\makelabel{rem:discete-time-is-pw:sig:article-simpl-wk-problem-raj-dahya}
    It is easy to see that the map ${T\mapsto (T^{n})_{n\in\ntrlzero}}$
    provides a topological isomorphism
    between
    the space, $\OpSpaceC{\HilbertRaum}$, of contractions under the $\topPW$-topology
    and
    the space, $\SpHs(\ntrlzero)$, of discrete-time contractive semigroups under the $\toplocWOT$-topology.
    Thus, studying the space of contractions under the $\topPW$-topology
    is equivalent to studying the space of discrete-time contractive semigroups.
\end{rem}

Now, what is clearly missing from \Cref{prop:spHs-and-spUHs-polish:sig:article-simpl-wk-problem-raj-dahya}
is a classification of the space
    $(\SpHs(M),\toplocWOT)$
for locally compact Polish monoids, $M$.
These topological spaces are the subject of the main result of this paper:

\begin{schattierteboxdunn}[backgroundcolor=leer,nobreak=true]
\begin{thm}%[Dahya]
\makelabel{thm:main-result:sig:article-simpl-wk-problem-raj-dahya}
    Let $\HilbertRaum$ denote a separable, infinite dimensional Hilbert space.
    The spaces of continuous- and discrete-time contractive $\C0$-semigroups over $\HilbertRaum$
    are Baire spaces under the $\toplocWOT$-topology.
\end{thm}
\end{schattierteboxdunn}

By \Cref{rem:discete-time-is-pw:sig:article-simpl-wk-problem-raj-dahya},
the result in the discrete case yields that the space of contractions under the $\topPW$-topology is a Baire space
(\cf \cite[Theorem~4.1]{eisnermaitrai2010typicalOperators}).

The claim in \Cref{thm:main-result:sig:article-simpl-wk-problem-raj-dahya} is that
    $(\SpHs(M),\toplocWOT)$ is a Baire space for $M\in\{\realsNonNeg,\ntrlzero\}$.
Working in this uniform context, our recipe for proving this is as follows:
(1) Classical \kurs{dilation} theorems allow us to approximate contractive semigroups via unitary semigroups,
thus yielding the \kurs{density} of $\SpUHs(M)$ within $(\SpHs(M),\toplocWOT)$;
(2) We develop a simple method to transfer certain properties,
including the property of being a Baire space,
from dense subspaces to larger spaces.
We utilise tools from descriptive set theory to classify topological properties via infinite games,
and prove that a slightly stronger condition than being a Baire space is also transferable.

\section[Density results]{Density results}
\label{sec:density-results}

\firstparagraph
We shall appeal to the following density results as a crucial step in solving the main problem.

\begin{lemm*}[Peller, 1981]
    The subset $\OpSpaceU{\HilbertRaum}$ is dense in $(\OpSpaceC{\HilbertRaum},\topPW)$.
\end{lemm*}

\begin{lemm*}[Król, 2009]
    Under the $\toplocWOT$-topology,
    all contractive $\C0$-semigroups over $\HilbertRaum$
    can be approximated
    via unitary $\C0$-semigroups over $\HilbertRaum$.
\end{lemm*}

See
    \cite[Theorem~1]{peller1981}
    and
    \cite[Theorem~2.1]{krol2009}
for proofs of these respective results.

\begin{rem}
    It turns out that in both of the above proofs,
    a key ingredient is the operator theoretic concept of \highlightTerm{dilation}.
    So, for the purposes of generalisation (\exempli to multiparameter semigroups),
    the approach in this paper might essentially depend on the existence of dilations.
    Whilst there exist dilation results for semigroups over $\realsNonNeg^{2}$
        (\cf
            \cite{slocinski1974}, %% TODO: finde die exakte Referenz!
            \cite[Theorem~2]{slocinski1982},
            and
            \cite[Theorem~2.3]{ptak1985}%
        ),
    for $\realsNonNeg^{d}$ with $d>2$,
    the existence of dilations is limited to certain semigroups
    (see \exempli \cite[Theorem~3.2]{ptak1985}).
\end{rem}

Placed in terms of the framework in this paper and noting \Cref{rem:discete-time-is-pw:sig:article-simpl-wk-problem-raj-dahya},
the above lemmata can be more uniformly summarised as follows:

\begin{lemm}
\makelabel{lemm:density-results-summarised:sec:density-results:sig:article-simpl-wk-problem-raj-dahya}
    The subspace, $\SpUHs(M)$, is dense in $(\SpHs(M),\toplocWOT)$ for $M\in\{\realsNonNeg,\ntrlzero\}$.
\end{lemm}

Note in the discrete case, that since all functions over $\ntrlzero$ are continuous,
the set of strongly continuous unitary semigroups over $\HilbertRaum$
is simply equal to the set of unitary semigroups over $\HilbertRaum$.

\section[Inheritance of properties from dense subspaces]{Inheritance of properties from dense subspaces}
\label{sec:intro}

\firstparagraph
In light of the above results, in order to solve the main problem,
it suffices to develop abstract methods to transfer properties from dense subsets to larger spaces.
We shall achieve this for the property of being a Baire space, as well as for the stronger condition of being a \kurs{Choquet} space.

\subsection[Definition of Baire and Choquet spaces via infinite games]{Definition of Baire and Choquet spaces via infinite games}
\label{sec:intro}

    \firstparagraph
    The following definitions and results can be found, for example, in
    \cite[\S{}8C--8E]{kech1994}.

    \begin{defn}[Choquet Game]
        For a topological space, $\RaumX$, the \highlightTerm{Choquet game},
        $\cal{G}_{\RaumX}$,
        is defined by two players, I and II, alternating
        and choosing non-empty open sets as follows:

            \begin{mathe}[mc]{rccccccl}
                \textup{I} &: &U_{0} &&U_{1} &&U_{2} &\cdots\\
                \textup{II} &: &&V_{0} &&V_{1} &&\cdots\\
            \end{mathe}

        \continueparagraph
        Such a \highlightTerm{run} of the game is \highlightTerm{valid},
        if and only if

            \begin{mathe}[mc]{rcccccccl}
                U_{0} &\supseteq &V_{0} &\supseteq &U_{1} &\supseteq &V_{1} &\supseteq &\ldots\\
            \end{mathe}

        \continueparagraph
        is satisfied. Player I wins such a run, exactly in case
            $\bigcap_{n\in\ntrlzero}U_{n}(=\bigcap_{n\in\ntrlzero}V_{n})=\leer$.
        Otherwise Player II wins.
        The \highlightTerm{strong Choquet game},
        $\cal{G}^{s}_{\RaumX}$,
        is defined similarly, except that on their $n$-th moves, $n\in\ntrlzero$,
        Player I additionally chooses some element $x_{n}\in U_{n}$
        and Player II must ensure that $x_{n}\in V_{n}$ holds.
    \end{defn}

    Loosely speaking, in these games, Player II attempts to construct an element or a subclass of elements,
    whilst Player I tries to frustrate his efforts by continually demanding that
    the element or subclass of elements realise ever more properties.
    Concretely, these games allow us to characterise topological concepts.
    By \cite[Theorem~8.11]{kech1994},
    we have:

    \begin{thm*}[Oxtoby]
        A topological space, $\RaumX$, is a Baire space if and only if Player I has no winning strategy in the game, $\cal{G}_{\RaumX}$.
    \end{thm*}

    Since the games are zero-sum,
    if Player II has a winning strategy in the Choquet game
    then Player I does not have one.
    This leads to the following strengthening of the concept of being a Baire space.

    \begin{defn}
        A topological space $\RaumX$ is called a \highlightTerm{(strong) Choquet space}
        if Player II has a winning strategy in the game, $\cal{G}_{\RaumX}$ (respectively in $\cal{G}^{s}_{\RaumX}$).
    \end{defn}

    It is easy to see, for a topological space, $\RaumX$, that a winning strategy for Player II in $\cal{G}^{s}_{\RaumX}$
    can be transferred to a winning strategy for Player II in $\cal{G}_{\RaumX}$
    (and excludes Player I from having a winning strategy in $\cal{G}_{\RaumX}$).
    Hence the following relations hold:

        \begin{mathe}[mc]{rcccl}
            \text{Strong Choquet} &\Longrightarrow &\text{Choquet} &\Longrightarrow &\text{Baire}.
        \end{mathe}

    \continueparagraph
    These implications are strict
    (see \cite[Exercises~8.13 and 8.15]{kech1994}).
    Being Choquet or strong Choquet characterises further useful topological properties
    (see
        \cite[Theorem~8.17]{kech1994}
    for a proof):

    \begin{thm*}[Choquet]
        Let $\RaumX$ be a separable metrisable space and,
        fixing a compatible metric $d$ on $\RaumX$,
        let $\RaumY$ be its unique metric completion.
        Then $\RaumX$ is comeagre in $\RaumY$ if and only if $\RaumX$ is Choquet.
        Furthermore, $\RaumX$ is Polish if and only if $\RaumX$ is strong Choquet.
    \end{thm*}

\subsection[The method of backwards inheritance]{The method of backwards inheritance}
\label{sec:intro}

    \firstparagraph
    We now develop a method to transfer properties from dense subspaces to their overlying spaces.

    \begin{defn}[Backwards inheritance]
        Call a property, $\Phi$, for topological subspaces
        \highlightTerm{backwards inheritable} exactly in case
        for all topological spaces $(\RaumX,\tau)$,
        if a dense subset $D\subseteq\RaumX$ exists,
        such that $(D,\tau\restr{D})$ satisfies $\Phi$,
        then $(\RaumX,\tau)$ satisfies $\Phi$.
    \end{defn}

    \begin{e.g.}
        Let $(\RaumX,\tau)$ be a topological space.
        Suppose a dense subset $D\subseteq\RaumX$ exists, such that $(D,\tau\restr{D})$ is separable.
        So \wrt $(D,\tau\restr{D})$ there exists a countable dense subset $D'\subseteq D$.
        Since density is a transitive property,\footnote{
            Let $U\subseteq\RaumX$ be non-empty and open.
            Then, by density, $U\cap D$ is non-empty and also open in $(D,\tau\restr{D})$.
            By density of $D'$ in $(D,\tau\restr{D})$,
            it follows that $U\cap D' =(U\cap D)\cap D'\neq\leer$.
        }
        $D'$ is dense in $\RaumX$.
        Hence $\RaumX$ is separable.
        It follows that separability is backwards inheritable.
    \end{e.g.}

    It shall now be shown that the property
        of being a Baire space
        as well as
        that of being a Choquet space,
    are backwards inheritable.

    \begin{schattierteboxdunn}[backgroundcolor=leer,nobreak=true]
    \begin{lemm}%[Dahya]
        \makelabel{lemm:backwardsinherit:baire:sig:article-simpl-wk-problem-raj-dahya}
        The property of being a Baire space is backwards inheritable.
    \end{lemm}
    \end{schattierteboxdunn}

        \begin{proof}
            Let $(\RaumX,\tau)$ be a topological space and let ${D\subseteq\RaumX}$ be dense.
            Suppose that $(D,\tau\restr{D})$ is a Baire space.
            Let $U_{n}\subseteq\RaumX$, $n\in\ntrl$
            be an arbitrary countable sequence of dense open subsets.
            It suffices to show that that $\bigcap_{n\in\ntrl}U_{n}$ is dense in $\RaumX$
            (\cf \cite[Definition~8.2]{kech1994}).

            First note, that if $U\subseteq\RaumX$ is a dense open subset,
            then clearly $U\cap D$ is also dense and open within $(D,\tau\restr{D})$.
            Thus, $U_{n}\cap D$ is a dense open subset in $(D,\tau\restr{D})$ for all $n\in\ntrl$.
            Since $(D,\tau\restr{D})$ is a Baire space,
            it follows that
            $C \colonequals (\bigcap_{n\in\ntrl}U_{n})\cap D=\bigcap_{n\in\ntrl}U_{n}\cap D$ is dense in $D$.
            Since density is transitive,
            it follows that $C$ and thereby $\bigcap_{n\in\ntrl}U_{n}$
            are dense in $(\RaumX,\tau)$.
            Thus $\RaumX$ is a Baire space.
        \end{proof}

    \begin{schattierteboxdunn}[backgroundcolor=leer,nobreak=true]
    \begin{lemm}%[Dahya]
        \makelabel{lemm:backwardsinherit:choquet:sig:article-simpl-wk-problem-raj-dahya}
        Being a Choquet space is backwards inheritable.
    \end{lemm}
    \end{schattierteboxdunn}

        \begin{proof}
            Let $(\RaumX,\tau)$ be a topological space and let ${D\subseteq\RaumX}$ be dense.
            Suppose that $(D,\tau\restr{D})$ is a Choquet space.
            Then Player II has a \highlightTerm{winning strategy}, $\tilde{\sigma}$, in the game $\cal{G}_{D}$.
            We now construct from $\tilde{\sigma}$ a strategy, $\sigma$, for Player II in the game, $\cal{G}_{\RaumX}$.
            In an arbitrary play of the game $\cal{G}_{\RaumX}$,
            for each $n\in\ntrlzero$ and each valid sequence

                \begin{mathe}[mc]{rcccccccccl}
                \eqtag[eq:1:\beweislabel]
                    U_{0} &\supseteq &V_{0}
                        &\supseteq &U_{1} &\supseteq &V_{1}
                        &\supseteq &\ldots
                        &\supseteq &U_{n}\\
                \end{mathe}

            \continueparagraph
            of choices up to Player I's $n$th move,
            define Player II's $n$th move under the strategy
            as $\sigma(\brkt{U_{0},V_{0},U_{1},V_{1},\ldots,U_{n}}) \colonequals V_{n}$,
            where $V_{n}$ is chosen as follows:

                \begin{kompaktenum}{\bfseries\scshape {Step} 1.}
                    %% STEP 1.
                    \item
                        Let $\tilde{U}_{i} \colonequals U_{i}\cap D$
                        for all ${i\in\{0,1,\ldots,n\}}$
                        and $\tilde{V}_{i} \colonequals V_{i}\cap D$
                        for all ${i\in\{0,1,\ldots,n-1\}}$.
                        By density of $D$, these are non-empty open subsets in $(D,\tau\restr{D})$,
                        and \eqcref{eq:1:\beweislabel}
                        implies

                            \begin{mathe}[mc]{rcccccccccl}
                                \tilde{U}_{0} &\supseteq &\tilde{V}_{0}
                                    &\supseteq &\tilde{U}_{1} &\supseteq &\tilde{V}_{1}
                                    &\supseteq &\ldots &\supseteq &\tilde{U}_{n}.\\
                            \end{mathe}
                    %% STEP 2.
                    \item
                        Since the sequence
                            $s \colonequals \brkt{\tilde{U}_{0},\tilde{V}_{0},\tilde{U}_{1},\tilde{V}_{1},\ldots,\tilde{U}_{n}}$
                        constitutes a valid sequence of previous plays of non-empty open
                        sets in $D$ for the game $\cal{G}_{D}$
                        before Player II's $n$th move in that game,
                        we may set
                        $\tilde{V}_{n} \colonequals \tilde{\sigma}(s)$,
                        \idest{} $\tilde{V}_{n}$ is Player II's next move
                        in $\cal{G}_{D}$ based on the strategy $\tilde{\sigma}$.
                        In particular, $\tilde{V}_{n}$ is a non-empty, open set in $(D,\tau\restr{D})$
                        and $\tilde{V}_{n}\subseteq\tilde{U}_{n}$,
                        by validity of the strategy $\tilde{\sigma}$.
                    %% STEP 3.
                    \item
                        Since $\tilde{V}_{n}$ is open in $D$, an open set $W\subseteq\RaumX$ exists, such that

                            \begin{mathe}[mc]{c}
                                \eqtag[eq:2:\beweislabel]
                                W\cap D = \tilde{V}_{n}\\
                            \end{mathe}

                        \continueparagraph
                        Set $V_{n} \colonequals W\cap U_{n}$, which again is open.
                        Since $U_{n}\supseteq\tilde{U}_{n}\supseteq\tilde{V}_{n}$,
                        by \eqcref{eq:2:\beweislabel} we have $V_{n}\supseteq\tilde{V}_{n}$.
                        Thus $\tilde{V}_{n}\subseteq V_{n}\cap D\subseteq W\cap D=\tilde{V}_{n}$,
                        so

                            \begin{mathe}[mc]{rclcl}
                                V_{n}\cap D &= &\tilde{V}_{n}
                                    &= &\tilde{\sigma}(\brkt{\tilde{U}_{0},\tilde{V}_{0},\tilde{U}_{1},\tilde{V}_{1},\ldots,\tilde{U}_{n}})\\
                            \end{mathe}

                        \continueparagraph
                        holds.
                        In particular, $V_{n}$ is a non-empty, open subset, contained in $U_{n}$,
                        and thus a valid $n$th move for Player II
                        in the game $\cal{G}_{\RaumX}$
                        based on the sequence in \eqcref{eq:1:\beweislabel}.
                \end{kompaktenum}

            It is now routine to check that $\sigma$ comprises a winning strategy for Player II in $\cal{G}_{\RaumX}$.
            By definition, this means that $\RaumX$ is Choquet.
        \end{proof}

    \begin{rem}
        \makelabel{rem:str-choq-not-backwardsinherit:sig:article-simpl-wk-problem-raj-dahya}
        Being strong Choquet ($\equiv$ Polish, by Choquet's theorem)
        is not backwards inheritable:
        Consider, \exempli, $\RaumX \colonequals D\cup(\{1\}\times\rtnl)$,
        where $D \colonequals [0,1)\times\reals$.
        Clearly, $D$ is dense in $\RaumX$ and Polish.
        If $\RaumX$ were Polish,
            then since $\{1\}\times\rtnl=\RaumX\cap(\{1\}\times\reals)$
            is a $G_{\delta}$-subset of $\RaumX$,
        it too would be Polish (%
            \cf
            \cite[Theorem~3.11]{kech1994}%
        ).
        Yet $\rtnl\cong\{1\}\times\rtnl$ is not even a Baire space.
    \end{rem}

\section[Proof of main result]{Proof of main result}
\label{sec:results}

\noindent
We may now prove \Cref{thm:main-result:sig:article-simpl-wk-problem-raj-dahya}.

\begin{proof}[of \Cref{thm:main-result:sig:article-simpl-wk-problem-raj-dahya}]
    By our setup (\cf
        \Cref{defn:standard-funct-spaces:sig:article-simpl-wk-problem-raj-dahya}
        and
        \Cref{rem:standard-funct-spaces:arbitrary-top-monoids:sig:article-simpl-wk-problem-raj-dahya}%
    ),
    we need to show that $(\SpHs(M),\toplocWOT)$ is a Baire space for $M\in\{\realsNonNeg,\ntrlzero\}$.
    Consider now the subspace, $\SpUHs(M)$, of $\topSOT$-continuous unitary semigroups.
    By \Cref{prop:spHs-and-spUHs-polish:sig:article-simpl-wk-problem-raj-dahya}
    the $\toplocWOT$- and $\toplocSOT$-topologies coincide on $\SpUHs(M)$
    and make this a Polish space.
    This immediately implies that $(\SpUHs(M),\toplocWOT)$ is a Choquet space and a Baire space.
    As per \Cref{lemm:density-results-summarised:sec:density-results:sig:article-simpl-wk-problem-raj-dahya},
        $\SpUHs(M)$ is dense in $(\SpHs(M),\toplocWOT)$.
    By backwards inheritance (\Cref{lemm:backwardsinherit:baire:sig:article-simpl-wk-problem-raj-dahya,lemm:backwardsinherit:choquet:sig:article-simpl-wk-problem-raj-dahya}),
    it follows that the larger space, $(\SpHs(M),\toplocWOT)$, is also a Choquet space and a Baire space.
\end{proof}

\section[Applications]{Applications}
\label{sec:applications}

\firstparagraph
We now focus squarely on the continuous case
and provide an application to rigidity phenomena.
From \Cref{thm:main-result:sig:article-simpl-wk-problem-raj-dahya},
we immediately obtain the following result:

\begin{prop}
\makelabel{prop:residual-unitary-iff-contractive:sig:article-simpl-wk-problem-raj-dahya}
    Let $\HilbertRaum$ be a separable infinite dimensional space.
    Consider the spaces of unitary and contractive $\C0$-semigroups over $\HilbertRaum$
    (which we denote $\SpUHs(\realsNonNeg)$ \respectively $\SpHs(\realsNonNeg)$),
    endowed with the topologies
    of uniform $\topWOT$-convergence on compact subsets of $\realsNonNeg$ (the $\toplocWOT$-topology),
    and of uniform $\topSOT$-convergence on compact subsets of $\realsNonNeg$ (the $\toplocSOT$-topology).
    Then the following

        \begin{kompaktenum}{\bfseries (i)}[\rtab]
            \item\punktlabel{1}
                $\{T\in\SpUHs(\realsNonNeg)\mid\Phi(T)\}$ is residual in $(\SpUHs(\realsNonNeg),\toplocSOT)$.
            \item\punktlabel{2}
                $\{T\in\SpUHs(\realsNonNeg)\mid\Phi(T)\}$ is residual in $(\SpUHs(\realsNonNeg),\toplocWOT)$.
            \item\punktlabel{3}
                $\{T\in\SpHs(\realsNonNeg)\mid\Phi(T)\}$ is residual in $(\SpHs(\realsNonNeg),\toplocWOT)$.
        \end{kompaktenum}

    \continueparagraph
    are equivalent,
    where $\Phi(\cdot)$ is any property defined on $\C0$-semigroups over $\HilbertRaum$.
\end{prop}

    \begin{proof}
        The equivalence of \punktcref{1} and \punktcref{2}
        holds by the trivial fact that
            $(\SpUHs(\realsNonNeg),\toplocSOT)$
            and
            $(\SpUHs(\realsNonNeg),\toplocWOT)$
        coincide topologically
        (see \Cref{prop:spHs-and-spUHs-polish:sig:article-simpl-wk-problem-raj-dahya}).
        To show \punktcref{2} $\Leftrightarrow$ \punktcref{3},
        since we know by \Cref{thm:main-result:sig:article-simpl-wk-problem-raj-dahya} that
            $(\SpHs(\realsNonNeg),\toplocWOT)$
        is a Baire space, it suffices to show that
            $\SpUHs(\realsNonNeg)$
        is comeagre in
            $(\SpHs(\realsNonNeg),\toplocWOT)$.
        To this end, first consider the spaces

            \begin{mathe}[mc]{rcccccl}
                \SpUHs(\realsNonNeg)
                    &\subseteq
                        &\SpHs(\realsNonNeg)
                    &\subseteq
                        &\SpCs(\realsNonNeg)
                    &\subseteq
                        &\SpCw(\realsNonNeg),\\
            \end{mathe}

        \continueparagraph
        endowed with the $\toplocWOT$-topology.
        Now we know that the outer two spaces are completely metrisable
        (%
            see
            \Cref{prop:op-func-spaces-are-polish:basic-func:sig:article-simpl-wk-problem-raj-dahya}
            and
            \Cref{prop:spHs-and-spUHs-polish:sig:article-simpl-wk-problem-raj-dahya}%
        ).
        By Alexandroff's lemma (\cf \cite[Theorem~3.11]{kech1994}),
        it follows that $\SpUHs(\realsNonNeg)$ is a $G_{\delta}$-subset of
            $\SpCw(\realsNonNeg)$
        and thus of $\SpHs(\realsNonNeg)$.
        Finally, by Król's approximation result (\cf \cite[Theorem~2.1]{krol2009}),
        we know that $\SpUHs(\realsNonNeg)$ is $\toplocWOT$-dense in $\SpHs(\realsNonNeg)$.
        Thus $\SpUHs(\realsNonNeg)$ is a comeagre subset in $(\SpHs(\realsNonNeg),\toplocWOT)$,
        and the proof is complete.
    \end{proof}

The proof of \Cref{prop:residual-unitary-iff-contractive:sig:article-simpl-wk-problem-raj-dahya}
also reveals the following result:

\begin{cor}
\makelabel{cor:unitary-C0-residual-in-contractive-C0:sig:article-simpl-wk-problem-raj-dahya}
    Let $\HilbertRaum$ be a separable infinite dimensional space.
    Then the subspace of unitary $\C0$-semigroups over $\HilbertRaum$
    is residual in the space of contractive $\C0$-semigroups over $\HilbertRaum$
    under the $\toplocWOT$-topology.
    Moreover, this subspace is of second category.
\end{cor}

By \Cref{prop:residual-unitary-iff-contractive:sig:article-simpl-wk-problem-raj-dahya},
the main application of our main result, is that
studying residual properties on the space of contractive $\C0$-semigroups
under the $\toplocWOT$-topology
is equivalent to
studying the same properties on the space of unitary $\C0$-semigroups
under the $\toplocSOT$-topology.
For example, when considering rigidity phenomena for semigroups
(see \cite[\S{}IV.3]{eisner2010buchStableOpAndSemigroups}),
we obtain \Cref{thm:application-rigidity:sig:article-simpl-wk-problem-raj-dahya}.
This is proved as follows:

    \begin{proof}[of \Cref{thm:application-rigidity:sig:article-simpl-wk-problem-raj-dahya}]
        The residuality result is known to hold
        for the unitary case under the $\toplocSOT$-topology
        (\cf \cite[Theorem~IV.3.20]{eisner2010buchStableOpAndSemigroups}).
        By
            \Cref{prop:residual-unitary-iff-contractive:sig:article-simpl-wk-problem-raj-dahya}%
            ~\eqcref{it:1:prop:residual-unitary-iff-contractive:sig:article-simpl-wk-problem-raj-dahya}%
            ~$\Longrightarrow$%
            ~\eqcref{it:3:prop:residual-unitary-iff-contractive:sig:article-simpl-wk-problem-raj-dahya},
        the residuality result can be transferred to the contractive case.
    \end{proof}

\begin{rem}
    The claims in \Cref{thm:application-rigidity:sig:article-simpl-wk-problem-raj-dahya}
    hold for the unitary case under both the
        $\toplocSOT$- and $\toplocWOT$-topologies,
    since these coincide.
    In the contractive case, not only are these topologies distinct,
    but the claim fails under the $\toplocSOT$-topology.
    To see this, consider the collection, ${\cal{S}\subseteq\SpHs(\realsNonNeg)}$,
    of semigroups unitarily equivalent to the \emph{backwards unilateral shift semigroup} on $\Lspace^{2}(\realsNonNeg^{2})$
    (\cf \cite[Definition~5.16]{eisnermaitrai2010typicalOperators}).
    Since $\cal{S}$ is $\toplocSOT$-residual in $\SpHs(\realsNonNeg)$
    (\cf \cite[Theorem~5.17]{eisnermaitrai2010typicalOperators})
    and each $S\in\cal{S}$ strongly converges asymptotically to the zero-operator,
    it follows that property \eqcref{it:2:thm:application-rigidity:sig:article-simpl-wk-problem-raj-dahya}
    in the statement of \Cref{thm:application-rigidity:sig:article-simpl-wk-problem-raj-dahya}
    does not hold residually \wrt the $\toplocSOT$-topology.
\end{rem}

These applications ultimately highlight that,
in the case of one-parameter $\C0$-semigroups under the $\toplocWOT$-topology,
there is no gain by working with contractive instead of unitary semigroups,
in terms of studying generic properties of semigroups.
Differences however do occur, when working under the $\toplocSOT$-topology.

%% BACKMATTER:

\subsection*{Acknowledgement}

\noindent
The author is grateful to Tanja Eisner for her encouragement to look into this problem,
and also to the referee for their constructive feedback,
which considerably helped improve the presentation of this paper.

%% LITERATURVERZEICHNIS:
\bibliographystyle{abbrv}
\def\bibname{References}
\bgroup

\egroup

\addresseshere
\end{document}